\documentclass[english]{amsart}
\usepackage{a4wide}

\newtheorem{theorem}{Theorem}[section]
\newtheorem{lemma}[theorem]{Lemma}

\newtheorem{corollary}[theorem]{Corollary}

\theoremstyle{definition}
\newtheorem{definition}[theorem]{Definition}
\newtheorem{example}[theorem]{Example}
\newtheorem{remark}[theorem]{Remark}

\numberwithin{equation}{section}

\newcommand{\abs}[1]{\lvert#1\rvert}

\usepackage{graphicx}

\usepackage[fixlanguage]{babelbib}
\selectbiblanguage{english}

\usepackage[shortlabels]{enumitem}

\usepackage{amsfonts,amssymb,amsbsy,amsmath,mathrsfs,amsthm}

\usepackage{commath}

\usepackage{hyperref}
\usepackage{url}

\hypersetup{pdfpagemode=UseNone, pdfstartview=FitH,
  colorlinks=true,urlcolor=red,linkcolor=blue,citecolor=black,
  pdftitle={Default Title, Modify},pdfauthor={Your Name},
  pdfkeywords={Modify keywords}}

\RequirePackage{babel}

\DeclareMathOperator{\BMO}{\mathrm{BMO}}
\DeclareMathOperator{\PBMO}{\mathrm{PBMO}}
\DeclareMathOperator{\dive}{div}

\providecommand{\abs}[1]{ \lvert#1  \rvert}
\providecommand{\norm}[1]{ \lVert#1  \rVert}

\newcommand{\dx}{\, d x}
\newcommand{\dy}{\, d y}
\newcommand{\dt}{\, d t}
\newcommand{\ds}{\, d s}
\newcommand{\dla}{\, d \lambda}
\newcommand{\lp}{\left (}
\newcommand{\rp}{\right )}

\def\Xint#1{\mathchoice
   {\XXint\displaystyle\textstyle{#1}}%
   {\XXint\textstyle\scriptstyle{#1}}%
   {\XXint\scriptstyle\scriptscriptstyle{#1}}%
   {\XXint\scriptscriptstyle\scriptscriptstyle{#1}}%
   \!\int}
\def\XXint#1#2#3{{\setbox0=\hbox{$#1{#2#3}{\int}$}
     \vcenter{\hbox{$#2#3$}}\kern-.5\wd0}}

\def\dashint{\Xint-}

\usepackage{adjustbox}
\usepackage{subcaption}
\usepackage{cite}
\makeatletter
\newcommand{\citecomment}[2][]{\citen{#2}#1\citevar}
\newcommand{\citeone}[1]{\citecomment{#1}}
\newcommand{\citetwo}[2][]{\citecomment[,~#1]{#2}}
\newcommand{\citevar}{\@ifnextchar\bgroup{;~\citeone}{\@ifnextchar[{;~\citetwo}{]}}}
\newcommand{\citefirst}{\@ifnextchar\bgroup{\citeone}{\@ifnextchar[{\citetwo}{]}}}

\makeatother

\begin{document}

\title{John--Nirenberg inequalities for parabolic BMO}

\author{Juha Kinnunen}
\address{Department of Mathematics, Aalto University, P.O. Box 11100, FI-00076 Aalto, Finland}
\email{juha.k.kinnunen@aalto.fi}
\thanks{The research was supported by the Academy of Finland.}

\author{Kim Myyryl\"ainen}
\address{Department of Mathematics, Aalto University, P.O. Box 11100, FI-00076 Aalto, Finland}
\email{kim.myyrylainen@aalto.fi}

\author{Dachun Yang}
\address{Laboratory of Mathematics and Complex Systems (Ministry of Education of China), 
School of Mathematical Sciences, Beijing Normal University, Beijing 100875, People's Republic of China}
\email{dcyang@bnu.edu.cn}


\begin{abstract}
We discuss a parabolic version of the space of functions of bounded mean oscillation related to a doubly nonlinear parabolic partial differential equation.
Parabolic John--Nirenberg inequalities, which give exponential decay estimates for the oscillation of a function, are shown in the natural geometry of the partial differential equation.
Chaining arguments are applied to change the time lag in the parabolic John--Nirenberg inequality.
We also show that the quasihyperbolic boundary condition is a necessary and sufficient condition for a global parabolic John--Nirenberg inequality.
Moreover, we consider John--Nirenberg inequalities with medians instead of integral averages and show that this approach gives the same class of functions as the original definition.
\end{abstract}

\maketitle

\section{Introduction}
Functions of bounded mean oscillation ($\BMO$) are essential in harmonic analysis and partial differential equations.
A particularly useful result is the John--Nirenberg lemma which gives an exponential decay estimate for the mean oscillation of a function in $\BMO$.
Functions of bounded mean oscillation and the John--Nirenberg lemma were first discussed in~\cite{john_original} and the corresponding time-dependent theory was initiated by Moser in~\cite{moser1964,moser1967}.
The proof of the parabolic John--Nirenberg lemma requires genuinely new ideas compared to the time independent case.
The main challenge is that the definition of parabolic $\BMO$ consists of two conditions on the mean oscillation of a function, one in the past and the other one in the future with a time lag between the estimates, see Definition~\ref{def.pbmo} below. 
Moreover, Euclidean cubes are replaced by rectangles that respect the natural geometry of the related parabolic partial differential equation.
Fabes and Garofalo~\cite{fabesgarofalo} gave a simpler proof for the parabolic John--Nirenberg lemma and the general approach of Aimar~\cite{aimar} applies in spaces of homogeneous type.
Mart\'{\i}n-Reyes and de la Torre~\cite{martin1994} studied one-sided $\BMO$ in the one-dimensional case.
Berkovits~\cite{berkovits2011,berkovits2012} discussed this approach in the higher dimensional case, but the geometry is not related to nonlinear parabolic partial differential equations.
In contrast with the extensive literature on the classical $\BMO$, there are few references to the corresponding parabolic theory. 

We discuss a parabolic $\BMO$ space tailored to a doubly nonlinear equation
\begin{equation}\label{eq.dnle}
\frac{\partial}{\partial t}(\lvert u\rvert^{p-2}u)-\dive(\lvert Du\rvert^{p-2}Du)=0,
\qquad 1<p<\infty.
\end{equation}
For $p=2$ we have the standard heat equation. Gradient and divergence are taken with respect to the spatial variable only. 
Observe that a solution to \eqref{eq.dnle} can be scaled, but constants cannot be added to a solution. 
The equation is nonlinear in the sense that the sum of two solutions is not a solution in general.
Our discussion applies to more general equations of the type
\[
\frac{\partial}{\partial t}(\lvert u\rvert^{p-2}u)-\dive A(x,t,u,Du)=0,
\]
where $A$ is a Caratheodory function 
that satisfies the structural conditions 
\[
A(x,t,u,Du) \cdot Du \ge C_0\lvert Du\rvert^p
\quad\text{and}\quad
\lvert A(x,t,u,Du)\rvert \le C_1\lvert Du\rvert^{p-1}
\]
for some positive constants $C_0$ and $C_1$. 
In the natural geometry of \eqref{eq.dnle}, we consider space-time rectangles 
where the time variable scales to the power $p$.
This is in accordance with the following scaling property: if $u(x,t)$ is a solution, so does $u(\alpha x,\alpha^p t)$ with $\alpha>0$.
Different values of $p$ lead to different $p$-geometries and different parabolic $\BMO$ spaces. 
For recent regularity results for the doubly nonlinear equation, we refer to B\"{o}gelein, Duzaar, Kinnunen and Scheven~\cite{bogelein2021b}, 
B\"{o}gelein, Duzaar and Liao~\cite{bogelein2021a},
Kuusi, Siljander and Urbano~\cite{kuusi2012a} and Kuusi, Laleoglu, Siljander and Urbano~\cite{kuusi2012b}.

The following scale and location invariant Harnack inequality for positive solutions to \eqref{eq.dnle} has been obtained by 
Moser~\cite{moser1964} for $p=2$ and by Trudinger~\cite{trudinger} for $1<p<\infty$.
See also Gianazza and Vespri~\cite{gianazzavespri} and Kinnunen and Kuusi~\cite{kinnunenkuusi}.
Assume that $u>0$ is a weak solution to \eqref{eq.dnle} in $\Omega_T=\Omega\times(0,T)$ and let $0<\gamma<1$ be a time lag.
There exist an exponent $\delta=\delta(n,p,\gamma)>0$  and constants $c_i=c_i(n,p,\gamma)>0$, $i=1,2,3$, such that
\begin{equation}
\label{eq.harnackproof}
\sup_{R^-(\gamma)} u  
\le c_1\lp\dashint_{2R^-(\gamma)} u^\delta\rp^{\frac1{\delta}}
\le c_2\lp\dashint_{2R^+(\gamma)} u^{-\delta}\rp^{-\frac1{\delta}}
\le c_3\inf_{R^+(\gamma)}u
\end{equation}
for every parabolic rectangle with $2R\subset\Omega_T$. 
See Definition~\ref{def_parrect} for the parabolic rectangles.
Harnack's inequality gives a scale and location invariant pointwise bound for a positive solution at a given time in terms of its values at later times.
This indicates that the parabolic rectangles in the $p$-geometry respect the natural geometry of the doubly nonlinear equation.
The time lag $\gamma>0$ is an unavoidable feature of the theory rather than a mere technicality. 
The fact that the result is not true with $\gamma=0$ can be seen from the heat kernel already when $p=2$, 
since Harnack's inequality does not hold on a given time slice.
The first and the last inequalities in \eqref{eq.harnackproof} are based on a successive application of Sobolev's inequality and energy estimates.
The remaining inequality follows from  the fact that a logarithm of a positive weak solution belongs to the parabolic $\BMO$ with a uniform estimate and a parabolic John--Nirenberg lemma as in \cite{moser1964} and \cite{trudinger}.
The proof in \cite{kinnunenkuusi} applies an abstract lemma of Bombieri instead of the parabolic John--Nirenberg lemma.
See also Kinnunen and Saari~\cite{kinnunenSaariMuckenhoupt} and Saari~\cite{localtoglobal}. 
The parabolic John--Nirenberg lemma with a general parameter $p$ has applications in the theory of parabolic Muckenhoupt weights 
in Kinnunen and Saari~\cite{kinnunenSaariMuckenhoupt, kinnunenSaariParabolicWeighted}.

We discuss several versions of the parabolic John--Nirenberg inequality in the $p$-geometry with $1<p<\infty$.
Theorem~\ref{reshetnyak} gives an exponential bound for the mean oscillation in terms of integral averages.
To our knowledge this result is new already for $p=2$ and generalizes the corresponding result for the standard $\BMO$ in \cite{john_original} to the parabolic case.
A more common version of the parabolic John--Nirenberg inequality is stated in Corollary~\ref{local_pJN}.
This generalizes the results of Moser~\cite{moser1964,moser1967} and Fabes and Garofalo~\cite{fabesgarofalo} to the $p$-geometry with $1<p<\infty$.
A general approach of Aimar~\cite{aimar} applies in metric measure spaces and also covers the $p$-geometry with $1<p<\infty$ by considering the parabolic metric 
\[
d((x,t),(y,s))=\max\{\|x-y\|_\infty,\lvert t-s\rvert^{\frac1p}\}.
\]
We prefer giving a direct and transparent proof in the $p$-geometry that also allows further investigation of the theory of parabolic $\BMO$.
The argument is based on a Calder\'on--Zygmund decomposition in the $p$-geometry.
The John--Nirenberg inequality implies that a parabolic $\BMO$ function is locally integrable to any positive power with reverse H\"older-type bounds, see Corollary~\ref{reverseHolder}.
Corollary~\ref{exp.int_upper} gives a stronger result which states that a parabolic $\BMO$ function is locally exponentially integrable with uniform estimates on parabolic rectangles.
A chaining argument in the proof of Theorem~\ref{global_pJN} shows that the size of the time lag can be changed in parabolic $\BMO$.
Parabolic chaining arguments have been previously studied by Saari~\cite{localtoglobal,forward-in-time} and our approach complements these techniques.
We also discuss the John--Nirenberg inequality up to the spatial boundary of a space-time cylinder by applying results of Saari~\cite{localtoglobal} and Smith and Stegenga~\cite{smithstegenga}.
In particular, we show that the quasihyperbolic boundary condition is a necessary and sufficient condition for a global parabolic John--Nirenberg inequality.
This extends some of the results in \cite{smithstegenga} to the parabolic setting.

John~\cite{johnmedian} observed that it is possible to relax the a priori local integrability assumption in the definition of $\BMO$
and to create a theory that applies to measurable functions. 
We extend this theory to the parabolic context.
This approach is based on the notion of median, for example, see Jawerth and Torchinsky~\cite{jawerth_torchinsky},  Poelhuis and Torchinsky~\cite{medians} and Str\"omberg~\cite{stromberg}.
It is remarkable that the John--Nirenberg inequality can be proved starting from this condition and as a consequence, these functions
are locally integrable to any positive power, see Theorem~\ref{local_pJN_m}.
Corollary~\ref{PBMOequivalent_m} shows that the parabolic $\BMO$ with medians coincides with the original definition of parabolic $\BMO$.

\section{Definition and properties of parabolic $\BMO$} 

The underlying space throughout is $\mathbb{R}^{n+1}=\{(x,t):x=(x_1,\dots,x_n)\in\mathbb R^n,t\in\mathbb R\}$.
Unless otherwise stated, constants are positive and the dependencies on parameters are indicated in the brackets.
The Lebesgue measure of a measurable subset $A$ of $\mathbb{R}^{n+1}$ is denoted by $\lvert A\rvert$.
A cube $Q$ is a bounded interval in $\mathbb R^n$, with sides parallel to the coordinate axes and equally long, that is,
$Q=Q(x,L)=\{y \in \mathbb R^n: \lvert y_i-x_i\rvert \leq L,\,i=1,\dots,n\}$
with $x\in\mathbb R^n$ and $L>0$. 
The point $x$ is the center of the cube and $L$ is the side length of the cube. 
Instead of Euclidean cubes, we work with the following collection of parabolic rectangles in $\mathbb{R}^{n+1}$.

\begin{definition}\label{def_parrect}
Let $1<p<\infty$, $x\in\mathbb R^n$, $L>0$ and $t \in \mathbb{R}$.
A parabolic rectangle centered at $(x,t)$ with side length $L$ is
\[
R = R(x,t,L) = Q(x,L) \times (t-L^p, t+L^p)
\]
and its upper and lower parts are
\[
R^+(\gamma) = Q(x,L) \times (t+\gamma L^p, t+L^p) 
\quad\text{and}\quad
R^-(\gamma) = Q(x,L) \times (t - L^p, t - \gamma L^p) ,
\]
where $-1 < \gamma < 1$ is called the time lag.
\end{definition}

Note that $R^-(\gamma)$ is the reflection of $R^+(\gamma)$ with respect to the time slice $\mathbb{R}^n \times \{t\}$.
The spatial side length of a parabolic rectangle $R$ is denoted by $l_x(R)=L$ and the time length by $l_t(R)=2L^p$.
For short, we write $R^\pm$ for $R^{\pm}(0)$.
The top of a rectangle $R = R(x,t,L)$ is $Q(x,L) \times\{t+L^p\}$
and the bottom is $Q(x,L) \times\{t-L^p\}$.
The $\lambda$-dilate of $R$  with $\lambda>0$ is denoted by $\lambda R = R(x,t,\lambda L)$.
We observe that the Lebesgue differentiation theorem holds on the collection of parabolic rectangles.

\begin{definition}
\label{conv_reg}
A sequence $(A_i)_{i\in\mathbb N}$ of measurable sets $A_i\subset\mathbb R^{n+1}$, $i\in\mathbb{N}$, converges regularly to a point $(x,t) \in \mathbb{R}^{n+1}$, if there exist a constant $c>0$ and a sequence  $(R_i)_{i\in\mathbb N}$ of parabolic rectangles $R_i$, $i\in\mathbb{N}$, such that  $\lvert R_i \rvert \to 0$ as $i\to\infty$, $A_i \subset R_i$, $(x,t) \in R_i$ and $\lvert A_i \rvert \leq \lvert R_i \rvert \leq c \lvert A_i \rvert$ for every $i \in\mathbb{N}$.
\end{definition}

The integral average of $f \in L^1(A)$ in measurable set $A\subset\mathbb{R}^{n+1}$, with $0<|A|<\infty$, is denoted by
\[
f_A = \dashint_A f \dx \dt = \frac{1}{\lvert A\rvert} \int_A f(x,t)\dx\dt .
\]
The Lebesgue differentiation theorem below can be proven in a similar way as in the classical case using a covering argument for parabolic rectangles.

\begin{lemma}
\label{LDT}
Let $f \in L^1_{\mathrm{loc}}(\mathbb{R}^{n+1})$. Then 
\[
\lim_{i \to \infty}  \dashint_{A_i} \lvert f(y,s) - f(x,t) \rvert \dy\ds = 0
\]
for almost every $(x,t) \in \mathbb{R}^{n+1}$, whenever $(A_i)_{i\in\mathbb N}$ is a sequence of measurable sets converging regularly to $(x,t)$.
\end{lemma}

The positive and the negative parts of a function $f$ are denoted by
\[
f_+ = \max\{ f, 0 \} 
\quad\text{and}\quad
f_- = - \min\{ f , 0 \} .
\]
Let $\Omega \subset \mathbb{R}^{n}$ be an open set and $T>0$.
A space-time cylinder is denoted by $\Omega_T=\Omega\times(0,T)$.
It is possible to consider space-time cylinders $\Omega\times(t_1,t_2)$ with $t_1<t_2$, but we focus on $\Omega_T$.

This section discusses basic properties of parabolic $\BMO$. 
We begin with the definition.
The differentials $\dx \dt$ in integrals are omitted in the sequel.

\begin{definition}
\label{def.pbmo}
Let $\Omega \subset \mathbb{R}^{n}$ be a domain, $T>0$, $ -1 < \gamma < 1$, $ -\gamma \leq \delta < 1$ and $0< q < \infty$.
A function $f \in L_{\mathrm{loc}}^q(\Omega_T)$ belongs to $\PBMO_{\gamma,\delta,q}^{+}(\Omega_T)$ if
\[
\norm{f}_{\PBMO_{\gamma,\delta,q}^{+}(\Omega_T)} 
= \sup_{R \subset \Omega_T} \inf_{c \in \mathbb{R}} \lp \dashint_{R^+(\gamma)} (f-c)_+^q 
+ \dashint_{R^-(\delta)} (f-c)_-^q \rp^\frac{1}{q} < \infty.
\]
If the condition above holds with the time axis reversed, then $f \in \PBMO_{\gamma,\delta,q}^{-}(\Omega_T)$.
\end{definition}

For $0 \leq \delta = \gamma < 1$, we abbreviate $\PBMO_{\gamma,q}^{+}(\Omega_T) = \PBMO_{\gamma,\delta,q}^{+}(\Omega_T)$.
In addition,
we shall write $\PBMO^{+}$ and $\norm{f}$ whenever parameters are clear from the context or are not of importance.
Observe that  $f \in L_{\mathrm{loc}}^q(\Omega_T)$ belongs to PBMO$_{\gamma,\delta,q}^{+}(\Omega_T)$ if and only if
for every parabolic rectangle $R$ there exists a constant $c\in\mathbb R$, that may depend on $R$, with
\[
\dashint_{R^+(\gamma)} (f-c)_+^q\le M
\quad\text{and}\quad
\dashint_{R^-(\delta)} (f-c)_-^q \le M,
\]
where  $M\in\mathbb R$ is a constant that is independent of $R$.

\begin{remark}\label{rem.logbmo}
Assume that $u>0$ is a weak solution to the doubly nonlinear equation in $\Omega_T$ and let $0<\gamma<1$.
By Kinnunen and Saari~\cite{kinnunenSaariMuckenhoupt} and Saari~\cite{localtoglobal}, we have 
\[
f=-\log u\in\PBMO^+_{\gamma,q}(\Omega_T),
\] 
with $q=(p-1)/2$ and
\[
\norm{f}=\|f\|_{\PBMO^+_{\gamma,q}}(\Omega_T)\le c(n,p,\gamma)<\infty.
\]
Observe that $0<q<1$ for $p<3$.
By H\"older's inequality, we may take $q=1$ for $p\ge3$.
Observe that the bound for the $\PBMO$ norm is independent of the solution.
\end{remark}

\begin{example}
Let $1<p<\infty$ and $0<\gamma<1$.
The function
\[
u(x,t)=t^{\frac{-n}{p(p-1)}}
e^{-\frac{p-1}p\left(\frac{|x|^p}{pt}\right)^{\frac1{p-1}}},
\,x\in\mathbb R^n,\, t>0,
\]
is a solution of the doubly nonlinear equation in the upper half space $\mathbb R^n\times(0,\infty)$.
By Remark~\ref{rem.logbmo}, we conclude that the function
\begin{align*}
f(x,t)
&=-\log u(x,t)
=-\log t^{\frac{-n}{p(p-1)}}
-\log e^{-\frac{p-1}p\left(\frac{|x|^p}{pt}\right)^{\frac1{p-1}}}\\
&=\tfrac{n}{p(p-1)}\log t
+\tfrac{p-1}p\left(\tfrac{|x|^p}{pt}\right)^{\frac1{p-1}},
\,x\in\mathbb R^n,\, t>0, 
\end{align*}
belongs to $\PBMO_{\gamma,q}^{+}(\mathbb R^n\times(0,\infty))$ with $q=(p-1)/2$.
Corollary~\ref{PBMOequivalent} below implies that $f$ belongs to $\PBMO_{\gamma,q}^{+}(\mathbb R^n\times(0,\infty))$ for every $0<q<\infty$.
\end{example}

The next lemma shows that for every parabolic rectangle $R$, there exists a constant $c_R$, depending on $R$, for which the infimum in the definition above is attained.
In the sequel, this minimal constant is denoted by $c_R$.

\begin{lemma}
\label{PBMO_constant}
Let $\Omega_T \subset \mathbb{R}^{n+1}$ be a space-time cylinder, $ -1 < \gamma < 1$, $ -\gamma \leq \delta < 1$ and $0< q < \infty$.
Assume that $f\in\PBMO_{\gamma,\delta,q}^{+}(\Omega_T)$.
Then for every parabolic rectangle $R\subset\Omega_T$, there exists a constant $c_R\in\mathbb R$, that may depend on $R$, such that 
\[
\dashint_{R^+(\gamma)} (f-c_R)_+^q  + \dashint_{R^-(\delta)} (f-c_R)_-^q = \inf_{c \in \mathbb{R}} \lp \dashint_{R^+(\gamma)} (f-c)_+^q + \dashint_{R^-(\delta)} (f-c)_-^q \rp .
\]
In particular, it holds that
\[
\sup_{R \subset \Omega_T} \lp \dashint_{R^+(\gamma)} (f-c_R)_+^q \ + \dashint_{R^-(\delta)} (f-c_R)_-^q  \rp^\frac{1}{q} 
= \norm{f}_{\PBMO_{\gamma,\delta,q}^{+}(\Omega_T)} .
\]
\end{lemma}

\begin{proof}
Let $R \subset \Omega_T$ be a parabolic rectangle. 
Consider a sequence $(c_i)_{i\in\mathbb N}$ of real numbers such that
\[
\dashint_{R^+(\gamma)} (f-c_i)_+^q  + \dashint_{R^-(\delta)} (f-c_i)_-^q
 < \inf_{c \in \mathbb{R}} \lp \dashint_{R^+(\gamma)} (f-c)_+^q  + \dashint_{R^-(\delta)} (f-c)_-^q  \rp + \frac{1}{2^i}
\]
for every $i \in \mathbb{N}$. 
Note that
\[
\dashint_{R^+(\gamma)} (f-c_i)_+^q + \dashint_{R^-(\delta)} (f-c_i)_-^q 
 \leq \dashint_{R^+(\gamma)} f_+^q  + \dashint_{R^-(\delta)} f_-^q  + 1 < \infty
\]
for every $i \in \mathbb{N}$, since $f \in L^q(\Omega_T)$.
If $0< q \leq 1$, then it holds that
\begin{align*}
(c_i)^q_- - \dashint_{R^+(\gamma)} f_-^q + (c_i)^q_+ - \dashint_{R^-(\delta)} f_+^q &\leq \dashint_{R^+(\gamma)} (c_i - f)_-^q + \dashint_{R^-(\delta)} (c_i - f)_+^q \\
&= \dashint_{R^+(\gamma)} (f-c_i)_+^q + \dashint_{R^-(\delta)} (f-c_i)_-^q \\
&\leq \dashint_{R^+(\gamma)} f_+^q  + \dashint_{R^-(\delta)} f_-^q  + 1 .
\end{align*}
This implies
\begin{align*}
\lvert c_i \rvert^q &= (c_i)^q_- + (c_i)^q_+ 
\leq \dashint_{R^+(\gamma)} ( f_+^q + f_-^q )  + \dashint_{R^-(\delta)} ( f_-^q + f_+^q )  + 1 \\
&= \dashint_{R^+(\gamma)} \lvert f \rvert^q  + \dashint_{R^-(\delta)} \lvert f \rvert^q  + 1 < \infty
\end{align*}
for every $i \in \mathbb{N}$.
On the other hand, if $1 < q < \infty$, we have
\begin{align*}
2^{1-q} (c_i)^q_- - \dashint_{R^+(\gamma)} f_-^q + 2^{1-q} (c_i)^q_+ - \dashint_{R^-(\delta)} f_+^q &\leq \dashint_{R^+(\gamma)} (c_i - f)_-^q + \dashint_{R^-(\delta)} (c_i - f)_+^q \\
&= \dashint_{R^+(\gamma)} (f-c_i)_+^q + \dashint_{R^-(\delta)} (f-c_i)_-^q \\
&\leq \dashint_{R^+(\gamma)} f_+^q  + \dashint_{R^-(\delta)} f_-^q  + 1 ,
\end{align*}
and thus
\begin{align*}
2^{1-q} \lvert c_i \rvert^q 
&= 2^{1-q} \lp (c_i)^q_- + (c_i)^q_+ \rp 
\leq \dashint_{R^+(\gamma)} ( f_+^q + f_-^q )  + \dashint_{R^-(\delta)} ( f_-^q + f_+^q )  + 1 \\
&= \dashint_{R^+(\gamma)} \lvert f \rvert^q  + \dashint_{R^-(\delta)} \lvert f \rvert^q  + 1 < \infty
\end{align*}
for every $i \in \mathbb{N}$.
This shows that the sequence $(c_i)_{i\in\mathbb N}$ is bounded. Therefore, there exists a subsequence $(c_{i_k})_{k\in\mathbb N} $ that converges to $c_R \in \mathbb{R}$. 
Then $(f-c_{i_k})_\pm$ converges uniformly to $(f-c_R)_\pm$ in $R$ as $k \to \infty$. This implies the convergence in $L^q(R)$, and thus
\begin{align*}
\inf_{c \in \mathbb{R}} \lp \dashint_{R^+(\gamma)} (f-c)_+^q  + \dashint_{R^-(\delta)} (f-c)_-^q \rp 
&= \lim_{k \to \infty} \lp \dashint_{R^+(\gamma)} (f-c_{i_k})_+^q + \dashint_{R^-(\delta)} (f-c_{i_k})_-^q  \rp \\
& = \dashint_{R^+(\gamma)} (f-c_R)_+^q + \dashint_{R^-(\delta)} (f-c_R)_-^q .
\end{align*}
This completes the proof.
\end{proof}

We list some properties of parabolic $\BMO$ below.
In particular, parabolic $\BMO$ is closed under addition and scaling by a positive constant. 
On the other hand, multiplication by negative constants reverses the time direction.

\begin{lemma}
\label{props_PBMO}
Let $ -1 < \gamma < 1$, $ -\gamma \leq \delta < 1$ and $0< q < \infty$.
Assume that $f$ and $ g$ belong to $\PBMO_{\gamma,\delta,q}^{+}(\Omega_T)$ and
let $\PBMO^{\pm}=\PBMO_{\gamma,\delta,q}^{\pm}(\Omega_T)$.
Then the following properties hold true.
\begin{enumerate}[(i), parsep=5pt, topsep=5pt]
    \item $\begin{aligned}[t]
        \norm{f+a}_{\PBMO^{+}} = \norm{f}_{\PBMO^{+}},\, a \in \mathbb{R}.
    \end{aligned}$
    \item $\begin{aligned}[t]
        \norm{f+g}_{\PBMO^{+}} \leq \max\{ 2^{\frac{1}{q}-1} , 2^{1-\frac{1}{q}} \} \lp \norm{f}_{\PBMO^{+}} + \norm{g}_{\PBMO^{+}} \rp.
    \end{aligned}$
    \item $\begin{aligned}[t]
        \norm{af}_{\PBMO^{+}} = \begin{cases} 
    \displaystyle
      a \norm{f}_{\PBMO^{+}}, &a \geq 0, \\
      \displaystyle
      -a \norm{f}_{\PBMO^{-}}, & a < 0.
   \end{cases}
    \end{aligned}$
    \item $\begin{aligned}
        \norm{\max\{f,g\}}_{\PBMO^{+}} &\leq \max\{1, 2^{\frac{1}{q} -1 }\} \lp \norm{f}_{\PBMO^{+}} + \norm{g}_{\PBMO^{+}} \rp, 
        \\
        \norm{\min\{f,g\}}_{\PBMO^{+}} &\leq \max\{1, 2^{\frac{1}{q} -1 }\} \lp \norm{f}_{\PBMO^{+}} + \norm{g}_{\PBMO^{+}} \rp.
    \end{aligned}$
\end{enumerate}
\end{lemma}

\begin{remark}
The constants in (ii) and (iv) can be avoided by considering the norm
\[
 \sup_{R \subset \Omega_T} \inf_{c \in \mathbb{R}}\left[ \lp \dashint_{R^+(\gamma)} (f-c)_+^q \rp^\frac{1}{q} 
+ \lp\dashint_{R^-(\delta)} (f-c)_-^q \rp^\frac{1}{q}\right],
\]
which is an equivalent norm with our definition. 
However, the current definition will be convenient in the proof of the John--Nirenberg lemma below.
\end{remark}

\begin{proof}[Proof of Lemma~\ref{props_PBMO}]
(i)
We observe that
\begin{align*}
&\inf_{c \in \mathbb{R}} \lp \dashint_{R^+(\gamma)} (f+a-c)_+^q + \dashint_{R^-(\delta)} (f+a-c)_-^q \rp^\frac{1}{q}\\
&\qquad= \inf_{c' \in \mathbb{R}} \lp \dashint_{R^+(\gamma)} (f-c')_+^q + \dashint_{R^-(\delta)} (f-c')_-^q \rp^\frac{1}{q}
\end{align*}
with $c' = c - a$.
Taking supremum over all parabolic rectangles $R \subset \Omega_T$ completes the proof.

(ii)
We note that
\[
\dashint_{R^\pm(\gamma) } (f + g -c_R^f - c_R^g)_\pm^q \leq \max\{ 1, 2^{q-1} \} \lp \dashint_{R^\pm(\gamma) } (f - c_R^f )_\pm^q + \dashint_{R^\pm(\gamma) } ( g - c_R^g)_\pm^q \rp
\]
for $c_R^f, c_R^g \in \mathbb{R}$.
This implies
\begin{align*}
&\inf_{c \in \mathbb{R}} \lp \dashint_{R^+(\gamma)} (f + g - c)_+^q + \dashint_{R^-(\delta)} (f + g - c)_-^q \rp^\frac{1}{q} \\
& \qquad  \leq \lp \dashint_{R^+(\gamma) } (f + g -c_R^f - c_R^g)_+^q + \dashint_{R^-(\delta) } (f + g -c_R^f - c_R^g)_-^q \rp^\frac{1}{q} \\
& \qquad  \leq \max\{ 1, 2^{1-\frac{1}{q}} \} \left( \dashint_{R^+(\gamma) } (f - c_R^f )_+^q + \dashint_{R^+(\gamma) } ( g - c_R^g)_+^q \right.\\
& \qquad \qquad\left.  + \dashint_{R^-(\delta) } (f - c_R^f )_-^q + \dashint_{R^-(\delta) } ( g - c_R^g)_-^q \right)^\frac{1}{q} \\
& \qquad  \leq \max\{ 2^{\frac{1}{q}-1} , 2^{1-\frac{1}{q}} \} \left( \dashint_{R^+(\gamma) } (f - c_R^f )_+^q + \dashint_{R^-(\delta) } (f - c_R^f )_-^q \right)^\frac{1}{q} \\
& \qquad \qquad  + \max\{ 2^{\frac{1}{q}-1} , 2^{1-\frac{1}{q}} \} \left( \dashint_{R^+(\gamma) } ( g - c_R^g)_+^q + \dashint_{R^-(\delta) } ( g - c_R^g)_-^q \right)^\frac{1}{q} .
\end{align*}
By taking supremum over all parabolic rectangles $R \subset \Omega_T$, we obtain
\[
\norm{f+g}_{\PBMO^{+}} \leq  \max\{ 2^{\frac{1}{q}-1} , 2^{1-\frac{1}{q}} \} \lp \norm{f}_{\PBMO^{+}} + \norm{g}_{\PBMO^{+}} \rp .
\]

(iii)
Observe that
\begin{align*}
&\inf_{c \in \mathbb{R}} \lp \dashint_{R^+(\gamma)} (af-c)_+^q + \dashint_{R^-(\delta)} (af-c)_-^q \rp^\frac{1}{q}\\
&\qquad= \inf_{c' \in \mathbb{R}} \lp \dashint_{R^+(\gamma)} (af-ac')_+^q + \dashint_{R^-(\delta)} (af-ac')_-^q \rp^\frac{1}{q} \\
& \qquad = \begin{cases} 
    \displaystyle
      a \inf_{c' \in \mathbb{R}} \lp \dashint_{R^+(\gamma)} (f-c')_+^q + \dashint_{R^-(\delta)} (f-c')_-^q \rp^\frac{1}{q}, &a \geq 0, \\
      \displaystyle
      -a \inf_{c' \in \mathbb{R}} \lp \dashint_{R^+(\gamma)} (f-c')_-^q + \dashint_{R^-(\delta)} (f-c')_+^q \rp^\frac{1}{q}, &a < 0,
   \end{cases}
\end{align*}
with $c' =c/a$. The claim follows from this observation.

(iv)
For the positive part, we have
\begin{align*}
&\int_{R^+(\gamma)} (\max\{f,g\} - \max\{c_R^f,c_R^g\})_+^q \\
& \qquad \leq \int_{R^+(\gamma) \cap \{ f \geq g \} } (\max\{f,g\} -c_R^f)_+^q  + \int_{R^+(\gamma) \cap \{ f < g \} } ( \max\{f,g\} -c_R^g)_+^q \\
&\qquad= \int_{R^+(\gamma) \cap \{ f \geq g \} } (f -c_R^f)_+^q + \int_{R^+(\gamma) \cap \{ f < g \} } ( g -c_R^g)_+^q \\
&\qquad \leq \int_{R^+(\gamma) } (f -c_R^f)_+^q + \int_{R^+(\gamma) } ( g -c_R^g)_+^q .
\end{align*}
For the negative part, we have
\[
\dashint_{R^-(\delta)} (\max\{f,g\} - \max\{c_R^f,c_R^g\})_-^q  \leq \dashint_{R^-(\delta)} (f -c_R^f)_-^q + \dashint_{R^-(\delta)} (g -c_R^g)_-^q .
\]
Hence, we obtain
\begin{align*}
&\inf_{c \in \mathbb{R}} \lp \dashint_{R^+(\gamma)} (\max\{f,g\} -c)_+^q  + \dashint_{R^-(\delta)} (\max\{f,g\} -c)_-^q  \rp^\frac{1}{q} \\
& \qquad \leq \lp \dashint_{R^+(\gamma)} (\max\{f,g\} - \max\{c_R^f,c_R^g\})_+^q  + \dashint_{R^-(\delta)} (\max\{f,g\} - \max\{c_R^f,c_R^g\})_-^q  \rp^\frac{1}{q} \\
& \qquad  \leq \lp \dashint_{R^+(\gamma) } (f -c_R^f)_+^q + \dashint_{R^-(\delta)} (f -c_R^f)_-^q + \dashint_{R^+(\gamma) } ( g -c_R^g)_+^q + \dashint_{R^-(\delta)} (g -c_R^g)_-^q \rp^\frac{1}{q} \\
& \qquad \leq \max\{1, 2^{\frac{1}{q} -1 }\} \lp \dashint_{R^+(\gamma) } (f -c_R^f)_+^q + \dashint_{R^-(\delta)} (f -c_R^f)_-^q \rp^\frac{1}{q} \\
& \qquad \qquad  + \max\{1, 2^{\frac{1}{q} -1 }\} \lp \dashint_{R^+(\gamma) } ( g -c_R^g)_+^q + \dashint_{R^-(\delta)} (g -c_R^g)_-^q \rp^\frac{1}{q} .
\end{align*}
By taking supremum over all parabolic rectangles $R \subset \Omega_T$, we conclude that
\[
\norm{\max\{f,g\}}_{\PBMO^{+}} \leq \max\{1, 2^{\frac{1}{q} -1 }\} \lp \norm{f}_{\PBMO^{+}} + \norm{g}_{\PBMO^{+}} \rp .
\]
The claim for $\min\{f,g\}$ follows similarly.
\end{proof}

\begin{remark}
Every function $f \in \PBMO_{\gamma,\delta,q}^{+}(\Omega_T)$ can be approximated pointwise by a sequence of bounded $\PBMO_{\gamma,\delta,q}^{+}(\Omega_T)$ functions, since the truncations
\[
f_k(x) = \min\{ \max\{ f(x), -k \}, k \} ,
\quad k\in\mathbb N,
\] 
belong to $\PBMO_{\gamma,\delta,q}^{+}(\Omega_T)$ with
\[
\norm{f_k}_{\PBMO_{\gamma,\delta,q}^{+}(\Omega_T)} \leq \max\{1, 2^{\frac{2}{q} -2 }\} \norm{f}_{\PBMO_{\gamma,\delta,q}^{+}(\Omega_T)}
\]
for every $k\in\mathbb N$, see (i) and (iv) of Lemma~\ref{props_PBMO}, and it holds that $f_k \to f$ pointwise and in $L^q(\Omega_T)$ as $k \to \infty$.
\end{remark}

\section{Parabolic John--Nirenberg inequalities} 
\label{section4}

This section discusses several versions of the John--Nirenberg inequality for parabolic $\BMO$.
We begin with a version where the exponential bound is given in terms of integral averages.
For short, we suppress the variables $(x,t)$ in the notation and, for example, write
\[
R^{+}(\alpha)\cap\{(f-c_{R})_+ > \lambda\}
=\{(x,t)\in R^{+}(\alpha):(f(x,t)-c_{R})_+ > \lambda\} 
\]
in the sequel.

\begin{theorem}
\label{reshetnyak}
Let $R \subset \mathbb{R}^{n+1}$ be a parabolic rectangle, $0 \leq \gamma < 1$, $\gamma < \alpha < 1$ and $0 < q \leq 1$.
Assume that $f \in\PBMO_{\gamma,q}^+(R)$ and let  $\norm{f}=\norm{f}_{\PBMO_{\gamma,q}^{+}(R)}$. 
Then there exist constants $c_R\in\mathbb R$, $A=A(n,p,q,\gamma,\alpha)$, $B=B(n,p,q,\gamma,\alpha)$ and $C=C(n,p,q,\gamma,\alpha)$ such that
\[
\lvert R^{+}(\alpha)\cap\{(f-c_{R})_+ > \lambda\} \rvert 
\leq e^{-B (\lambda/\norm{f})^q } \frac{A}{\norm{f}^q} \int_{R^{+}(\gamma)}(f-c_{R})_+^q 
\]
and
\[
\lvert R^{-}(\alpha)\cap\{(f-c_{R})_- > \lambda \} \rvert 
\leq e^{-B(\lambda/\norm{f})^q } \frac{A}{\norm{f}^q} \int_{R^{-}(\gamma)} (f-c_{R})_-^q
\]
for every $\lambda \geq C \norm{f}$.
\end{theorem}

\begin{proof}
Let $R_0=R=R(x_0,t_0,L) = Q(x_0,L) \times (t_0-L^p, t_0+L^p)$.
By considering the function $f(x+x_0,t+t_0)$, we may assume that the center of $R_0$ is the origin, that is, $R_0 = Q(0,L) \times (-L^p, L^p)$.
By (i) and (iii) of Lemma~\ref{props_PBMO}, we may assume that $c_{R_0} = 0$ and $\norm{f}^q = \tfrac{1}{2}(1-\alpha)/(1-\gamma)$.
We note that it is sufficient to prove the first inequality of the theorem since
the second inequality follows from the first one by applying it to the function $-f(x,-t)$.

We claim that
\[
\lvert R^{+}_0(\alpha)\cap\{f_+ > \lambda \} \rvert 
\leq e^{-B(\lambda/\norm{f})^q } \frac{A}{\norm{f}^q} \int_{R^{+}_0(\gamma)} f_+^q 
\]
for every $\lambda \geq C \norm{f}$.
Let $m$ be the smallest integer with
$3 + \alpha \leq 2^{pm+1} (\alpha - \gamma)$, 
that is,
\[
\frac{1}{p} \log_2\lp\frac{3+ \alpha}{2 (\alpha - \gamma)}\rp \leq m < \frac{1}{p} \log_2\lp\frac{3+ \alpha}{2 (\alpha - \gamma)}\rp + 1 .
\]
Let $S^+_0 = R^+_0(\alpha) = Q(0,L) \times (\alpha L^p,L^p) $. The time length of $S^+_0$ is $l_t(S^+_0) = (1-\alpha)L^p$.
We partition $S^+_0$ by dividing each spatial edge $[-L,L]$ into $2^m$ equally long intervals.
If
\[
\frac{l_t(S_{0}^+)}{\lfloor 2^{pm} \rfloor} < \frac{(1-\alpha)L^p}{2^{pm}},
\]
we divide the time interval of $S^+_0$ into $\lfloor 2^{pm} \rfloor$ equally long intervals. 
Otherwise, we divide the time interval of $S^+_0$ into $\lceil 2^{pm} \rceil$ equally long intervals.
We obtain subrectangles $S^+_1$ of  $S^+_0$ with spatial side length 
$l_x(S^+_1)=l_x(S^+_0)/2^m = L / 2^m$ and time length either 
\[
l_t(S^+_1)=\frac{l_t(S^+_0)}{\lfloor 2^{pm} \rfloor} 
=\frac{(1-\alpha)L^p}{\lfloor 2^{pm} \rfloor} 
\quad\text{or}\quad
l_t(S^+_1)=\frac{(1-\alpha)L^p}{\lceil 2^{pm} \rceil}.
\]
For every $S^+_1$, there exists a unique rectangle $R_1$ with spatial side length $l_x = L / 2^{m}$ and time length $l_t = 2 L^p / 2^{mp}$
such that $R_1$ has the same top as $S^+_1$.
We select those rectangles $S^+_1$ for which $\lambda < c_{R_1}$ and denote the obtained collection by $\{ S^+_{1,j} \}_j$.
If $\lambda \geq c_{R_1}$, we subdivide $S^+_1$ in the same manner as above
and select all those subrectangles $S^+_2$ for which $\lambda < c_{R_2}$ to obtain family $\{ S^+_{2,j} \}_j$.
We continue this selection process recursively.
At the $i$th step, we partition unselected rectangles $S^+_{i-1}$ by dividing each spatial side into $2^m$ equally long intervals. 
If 
\begin{equation}
\label{JNproof_eq1}
\frac{l_t(S_{i-1}^+)}{\lfloor 2^{pm} \rfloor} < \frac{(1-\alpha)L^p}{2^{pmi}},
\end{equation}
we divide the time interval of $S^+_{i-1}$ into $\lfloor 2^{pm} \rfloor$ equally long intervals. 
If
\begin{equation}
\label{JNproof_eq2}
\frac{l_t(S_{i-1}^+)}{\lfloor 2^{pm} \rfloor} \geq \frac{(1-\alpha)L^p}{2^{pmi}},
\end{equation}
we divide the time interval of $S^+_{i-1}$ into $\lceil 2^{pm} \rceil$ equally long intervals.
We obtain subrectangles $S^+_i$. 
For every $S^+_i$, there exists a unique rectangle $R_i$ with spatial side length $l_x = L / 2^{mi}$ and time length $l_t = 2 L^p / 2^{pmi}$
such that $R_i$ has the same top as $S^+_i$.
Select those $S^+_i$ for which $\lambda < c_{R_i}$ and denote the obtained collection by $\{ S^+_{i,j} \}_j$.
If $\lambda \geq c_{R_i}$, we continue the selection process in $S^+_i$.
In this manner we obtain a collection $\{S^+_{i,j} \}_{i,j}$ of pairwise disjoint rectangles.

Observe that if \eqref{JNproof_eq1} holds, then we have
\[
l_t(S_i^+) = \frac{l_t(S^+_{i-1})}{\lfloor 2^{pm} \rfloor} < \frac{(1-\alpha)L^p}{2^{pmi}}.
\]
On the other hand, if \eqref{JNproof_eq2} holds, then
\[
l_t(S_i^+) = \frac{l_t(S^+_{i-1})}{\lceil 2^{pm} \rceil} \leq \frac{l_t(S^+_{i-1})}{2^{pm}} \leq \dots \leq \frac{(1-\alpha)L^p}{2^{pmi}} .
\]
This gives an upper bound 
\[
l_t(S_i^+) \leq \frac{(1-\alpha)L^p}{2^{pmi}}
\]
for every $S_i^+$.

\begin{figure}[t!]
    \centering
    \includegraphics[trim=0cm 2cm 0cm 0.75cm, width=1\textwidth]{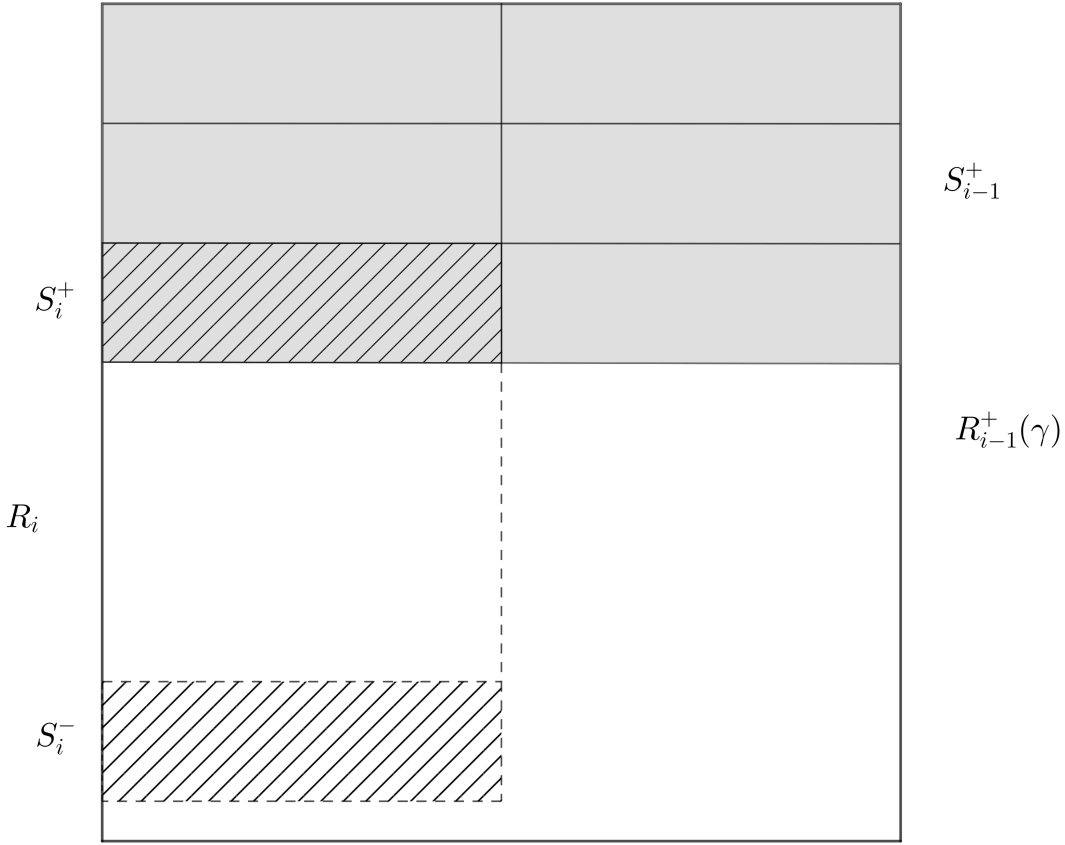}
    \caption{Schematic picture of the decomposition at the $i$th step.}
    \label{fig:decomp}
\end{figure}

Suppose that \eqref{JNproof_eq2} is satisfied at the $i$th step.
Then we have a lower bound for the time length of $S_i^+$, since
\[
l_t(S^+_i) = \frac{l_t(S_{i-1}^+)}{\lceil 2^{pm} \rceil} \geq \frac{\lfloor 2^{pm} \rfloor}{\lceil 2^{pm} \rceil} \frac{(1-\alpha)L^p}{2^{pmi}} \geq \frac{1}{2} \frac{(1-\alpha)L^p}{2^{pmi}} .
\]
On the other hand, if \eqref{JNproof_eq1} is satisfied, then
\[
l_t(S^+_i) = \frac{l_t(S_{i-1}^+)}{\lfloor 2^{pm} \rfloor} \geq \frac{l_t(S_{i-1}^+)}{ 2^{pm}}.
\]
In this case, \eqref{JNproof_eq2} has been satisfied at an earlier step $i'$ with $i'< i$.
We obtain
\[
l_t(S^+_i) \geq \frac{l_t(S_{i-1}^+)}{ 2^{pm}} \geq \dots \geq \frac{l_t(S_{i'}^+)}{ 2^{pm(i-i')}} \geq \frac{1}{2} \frac{(1-\alpha)L^p}{ 2^{pmi}}
\]
by using the lower bound for $S_{i'}^+$.
Thus, we have 
\[
\frac{1}{2} \frac{(1-\alpha)L^p}{2^{pmi}} \leq l_t(S^+_i) \leq \frac{(1-\alpha)L^p}{2^{pmi}}
\]
for every $S^+_i$.
By using the lower bound for the time length of  $S^+_i$ and the choice of $m$, we observe that
\begin{align*}
l_t(R_i) - l_t(S^+_i) 
&\leq \frac{2 L^p}{2^{pmi}} - \frac{1}{2} \frac{(1-\alpha)L^p}{2^{pmi}} 
= \frac{1}{2} \frac{L^p}{2^{pmi}} (3+ \alpha) \\
&\leq \frac{(\alpha-\gamma) L^p}{2^{pm(i-1)}} 
= \frac{(1-\gamma)L^p}{2^{pm(i-1)}} - \frac{(1-\alpha)L^p}{2^{pm(i-1)}} \\
&\leq l_t(R^+_{i-1}(\gamma)) - l_t(S^+_{i-1}) .
\end{align*}
This implies
\begin{equation}
\label{subset}
R_{i} \subset R^+_{i-1}(\gamma)
\end{equation}
for a fixed rectangle $S^+_{i-1}$ and for every subrectangle $S^+_{i} \subset S^+_{i-1}$ (see Fig.~\ref{fig:decomp}).
By the construction of the subrectangles $S^+_i$, we have
\begin{equation}
\label{measure1}
2^{nm} \lfloor 2^{pm} \rfloor \lvert S^+_i \rvert 
\leq \lvert S^+_{i-1} \rvert \leq 2^{nm} \lceil 2^{pm} \rceil \lvert S^+_i \rvert 
\end{equation}
and
\begin{equation}
\label{measure2}
\frac{1}{2} \frac{1-\alpha}{1-\gamma} \lvert R^+_i(\gamma) \rvert 
\leq \lvert S^+_i \rvert \leq \frac{1-\alpha}{1-\gamma} \lvert R^+_i(\gamma) \rvert .
\end{equation}

We have a collection $\{ S^+_{i,j} \}_{i,j}$ of pairwise disjoint rectangles. 
However, the rectangles in the corresponding collection $\{ S^-_{i,j} \}_{i,j}$ may overlap. 
Thus, we replace it by a subfamily $\{ \widetilde{S}^-_{i,j} \}_{i,j}$ of pairwise disjoint rectangles, which is constructed in the following way.
At the first step, choose $\{ S^-_{1,j} \}_{j}$ and denote it by $\{ \widetilde{S}^-_{1,j} \}_j$. 
Then consider the collection $\{ S^-_{2,j} \}_{j}$ where each $S^-_{2,j}$ either intersects some $\widetilde{S}^-_{1,j}$ or does not intersect any $\widetilde{S}^-_{1,j}$. 
Select the rectangles $S^-_{2,j}$ that do not intersect any $\widetilde{S}^-_{1,j}$, and denote the obtained collection by $\{ \widetilde{S}^-_{2,j} \}_j$.
At the $i$th step, choose those $S^-_{i,j}$ that do not intersect any previously selected $\widetilde{S}^-_{i',j}$, $i' < i$.
Hence, we obtain a collection $\{ \widetilde{S}^-_{i,j} \}_{i,j}$ of pairwise disjoint rectangles.
Observe that for every $S^-_{i,j}$ there exists $\widetilde{S}^-_{i',j}$ with $i' < i$ such that
\begin{equation}
\label{plussubset}
\text{pr}_x(S^-_{i,j}) \subset \text{pr}_x(\widetilde{S}^-_{i',j}) \quad \text{and} \quad \text{pr}_t(S^-_{i,j}) \subset 3 \text{pr}_t(\widetilde{S}^-_{i',j}) .
\end{equation}
Here pr$_x$ denotes the projection to $\mathbb R^n$ and pr$_t$ denotes the projection to the time axis.

Rename $\{ S^+_{i,j} \}_{i,j}$ and $\{ \widetilde{S}^-_{i,j} \}_{i,j}$ as $\{ S^+_{i} \}_{i}$ and $\{ \widetilde{S}^-_{j} \}_j$, respectively.
Let $S(\lambda) = \bigcup_i S^+_i$.
Note that $S^+_i$ is spatially contained in $S^-_i$, that is, $\text{pr}_x S^+_i\subset \text{pr}_x S^-_i$.
In the time direction, we have
\begin{equation}
\label{minusplussubset}
\text{pr}_t(S^+_i) \subset \text{pr}_t(R_i) 
\subset \frac{7 + \alpha}{1-\alpha} \text{pr}_t(S^-_i) ,
\end{equation}
since
\[
\lp \frac{7 + \alpha}{1-\alpha} + 1 \rp \frac{l_t(S^-_i)}{2} \geq  \frac{8}{1-\alpha}  \frac{(1-\alpha)L^p}{2^{pmi+2}} = \frac{2L^p}{2^{pmi}} = l_t(R_i) .
\]
Therefore, by~\eqref{plussubset} and~\eqref{minusplussubset}, it holds that
\begin{equation}
\label{start}
\lvert S(\lambda) \rvert = \Big\lvert \bigcup_i S^+_i \Big\rvert \leq c_1 \sum_j \lvert \widetilde{S}^-_j \rvert 
\quad\text{with}\quad
c_1 = 3\frac{7+\alpha}{1-\alpha}.
\end{equation}

Let $\lambda > \delta > 0$ and consider the collection $\{S^+_k\}_k$ for $\delta$.
Then every $S^+_i$ is contained in a unique $S^+_k$.
Let $\mathcal{J}_k = \{ j \in \mathbb{N}: \widetilde{S}^+_j \subset S^+_k \}$.
Using \eqref{measure2}, we have
\begin{align*}
\lambda^q &< c_{R_j}^q \leq \dashint_{\widetilde{S}^-_j} (c_{R_j} - f)^q_+  + \dashint_{\widetilde{S}^-_j} f^q_+ \\
&\leq 2 \frac{1-\gamma}{1-\alpha} \dashint_{R^-_j(\gamma)} (c_{R_j} - f)^q_+ + \dashint_{\widetilde{S}^-_j} f^q_+ \\
&\leq 2 \frac{1-\gamma}{1-\alpha} \norm{f}^q + \dashint_{\widetilde{S}^-_j} f^q_+ = 1 + \dashint_{\widetilde{S}^-_j} f^q_+ 
\end{align*}
for every $\widetilde{S}^-_j$.
By summing over $j$, we obtain
\begin{equation}
\label{decomp1}
(\lambda^q - 1) \sum_j \lvert \widetilde{S}^-_j \rvert \leq \sum_j \int_{\widetilde{S}^-_j} f^q_+  = \sum_k \sum_{j \in \mathcal{J}_k} \int_{\widetilde{S}^-_j} f^q_+ .
\end{equation}
Let $k \in \mathbb{N}$. We have $\widetilde{S}^+_j \subset S^+_k$ for every $j \in \mathcal{J}_k$, where $S^+_k$ was obtained by subdividing a previous $S^+_{k^-}$ for which $a_{R_{k^-}} \leq \delta$.
Hence, it holds that $\widetilde{S}^-_j \subset R_k$ for every $j \in \mathcal{J}_k$.
By \eqref{subset}, it follows that $\widetilde{S}^-_j \subset R^+_{k^-}(\gamma)$ for every $j \in \mathcal{J}_k$.
Since $\widetilde{S}^-_j$ are pairwise disjoint, by applying \eqref{measure1} together with \eqref{measure2}, we arrive at
\begin{align*}
\sum_{j \in \mathcal{J}_k} \int_{\widetilde{S}^-_j} f^q_+ &\leq \sum_{j \in \mathcal{J}_k} \int_{\widetilde{S}^-_j} (f - a_{R_{k^-}} + \delta)^q_+ \\
&\leq \sum_{j \in \mathcal{J}_k} \int_{\widetilde{S}^-_j} (f - a_{R_{k^-}})^q_+ + \sum_{j \in \mathcal{J}_k} \int_{\widetilde{S}^-_j} \delta^q \\
&\leq \int_{R^+_{k^-}(\gamma)} (f - a_{R_{k^-}})^q_+ + \delta^q \sum_{j \in \mathcal{J}_k} \lvert \widetilde{S}^-_j \rvert \\
&\leq \frac{1}{2} \frac{1-\alpha}{1-\gamma} \lvert R^+_{k^-}(\gamma) \rvert + \delta^q \sum_{j \in \mathcal{J}_k} \lvert \widetilde{S}^-_j \rvert \\
&\leq \lvert S^+_{k^-} \rvert + \delta^q \sum_{j \in \mathcal{J}_k} \lvert \widetilde{S}^-_j \rvert \\
&\leq 2^{nm} \lceil 2^{pm} \rceil \lvert S^+_k \rvert + \delta^q \sum_{j \in \mathcal{J}_k} \lvert \widetilde{S}^-_j \rvert \\
&\leq c_2 \lvert S^+_k \rvert + \delta^q \sum_{j \in \mathcal{J}_k} \lvert \widetilde{S}^-_j \rvert ,
\end{align*}
where
\begin{align*}
2^{nm} \lceil 2^{pm} \rceil 
\leq 2^{nm} 2^{pm+1}
&\leq 2^{1+(n+p) \bigl( \frac{1}{p} \log_2\lp\frac{3+\alpha}{2(\alpha- \gamma)}\rp + 1 \bigr)} \\
&= 2^{1+n+p} \lp\frac{3+\alpha}{2(\alpha- \gamma)}\rp^{1 + \frac{n}{p}}
= c_2 .
\end{align*}
By summing over $k$ and applying \eqref{decomp1}, we obtain
\[
(\lambda^q - 1) \sum_j \lvert \widetilde{S}^-_j \rvert \leq c_2 \sum_k \lvert S^+_k \rvert + \delta^q \sum_{j} \lvert \widetilde{S}^-_j \rvert .
\]
Thus, we have
\[
\sum_j \lvert \widetilde{S}^-_j \rvert \leq \frac{ c_2}{\lambda^q - \delta^q - 1} \sum_k \lvert S^+_k \rvert 
\]
for every $\lambda^q > \delta^q + 1$.
By \eqref{start}, we obtain
\[
\lvert S(\lambda) \rvert \leq \frac{c_1 c_2}{\lambda^q - \delta^q - 1} \lvert S(\delta) \rvert 
\]
for every $\lambda^q > \delta^q + 1$.
By setting $a = 2c_1 c_2 + 1$ and replacing $\lambda^q$ and $\delta^q$ by $\lambda + a$ and $\lambda$, respectively, we have
\begin{equation}
\label{iteration}
\lvert S((\lambda + a)^\frac{1}{q}) \rvert \leq \frac{1}{2} \lvert S(\lambda^\frac{1}{q}) \rvert .
\end{equation}
Assume that $\lambda \geq a$. Then there exists an integer $N \in \mathbb{Z}_+$ such that $Na \leq \lambda < (N+1)a$.
By iterating~\eqref{iteration}, we arrive at
\begin{align*}
\lvert S(\lambda^\frac{1}{q}) \rvert 
&\leq \lvert S((Na)^\frac{1}{q}) \rvert \leq \frac{1}{2^{N-1}} \lvert S(a^\frac{1}{q}) \rvert \\
&\leq 2^{-\frac{\lambda}{a}+2} \lvert S(a^\frac{1}{q}) \rvert = 4 e^{-\frac{\lambda}{2c_1 c_2 +1} \log 2} \lvert S(a^\frac{1}{q}) \rvert .
\end{align*}
Applying \eqref{start}, \eqref{decomp1} and \eqref{subset}, we get
\begin{align*}
\lvert S(a^\frac{1}{q}) \rvert \leq \frac{c_1}{a-1} \int_{R^{+}_0(\gamma)} f_+^q = \frac{1}{4c_2}
\frac{1-\alpha}{1-\gamma} \frac{1}{\norm{f}^q} \int_{R^{+}_0(\gamma)} f_+^q .
\end{align*}
This implies
\[
\lvert S(\lambda) \rvert 
\leq e^{-2B(\lambda/\norm{f})^q } \frac{A}{\norm{f}^q} \int_{R^{+}_0(\gamma)} f_+^q 
\]
for every $\lambda \geq a^\frac{1}{q}$ with 
\[
A = \frac{1}{c_2} \frac{1-\alpha}{1-\gamma}
\quad\text{and}\quad 
B = \frac{1}{4} \frac{1-\alpha}{1-\gamma} \frac{\log 2}{2 c_1 c_2 + 1}.
\]

If $(x,t) \in S^+_0 \setminus S(\lambda)$, then there exists a sequence $\{S^+_l\}_{l\in\mathbb N}$
of subrectangles containing $(x,t)$ such that $c_{R_l} \leq \lambda $ and $\lvert S^+_l \rvert \to 0$ as $l \to \infty$.
This implies
\[
\dashint_{S^+_l} f^q_+ \leq \dashint_{S^+_l} (f-c_{R_l})^q_+ + \lambda^q \leq 1 + \lambda^q .
\]
The Lebesgue differentiation theorem (Lemma~\ref{LDT}) implies that
$f(x,t)^q_+ \leq 1 + \lambda^q$ for almost every $(x,t) \in S^+_0 \setminus S(\lambda)$.
It follows that
\[
\{ (x,t) \in S^+_0 : f(x,t)^q_+ > 1 + \lambda^q \} \subset S(\lambda)
\]
up to a set of measure zero.
Given $\lambda \geq 2^\frac{1}{q}$, we have $\lambda^q \geq 1 + \frac{\lambda^q}{2} $. We conclude that
\begin{align*}
\lvert S^+_0 \cap \{ f_+ > \lambda \} \rvert 
&= \lvert S^+_0 \cap \{ f_+^q > \lambda^q \} \rvert 
\leq \lvert S^+_0 \cap \{ f_+^q> 1 + \tfrac{\lambda^q}{2} \} \rvert 
\leq \lvert S(\lambda / 2^\frac{1}{q} ) \rvert \\
&\leq e^{-B(\lambda/\norm{f})^q} \frac{A}{\norm{f}^q} \int_{R^{+}_0(\gamma)} f_+^q 
\end{align*}
for every $\lambda \geq (2 a)^\frac{1}{q} = \bigl(4a(1-\gamma)/(1-\alpha)\bigr)^\frac{1}{q} \norm{f} = C \norm{f} $.
This completes the proof.
\end{proof}

As a corollary of Theorem~\ref{reshetnyak}, we obtain a more standard version of the parabolic John--Nirenberg inequality.

\begin{corollary}
\label{local_pJN}
Let $R \subset \mathbb{R}^{n+1}$ be a parabolic rectangle, $0 \leq \gamma < 1$, $\gamma < \alpha < 1$ and $0 < q \leq 1$.
Assume that $f \in\PBMO_{\gamma,q}^+(R)$ and let  $\norm{f}=\norm{f}_{\PBMO_{\gamma,q}^{+}(R)}$. 
Then there exist constants $c_{R}\in\mathbb R$, $A=A(n,p,q,\gamma,\alpha)$ and $B=(n,p,q,\gamma,\alpha)$ such that
\[
\lvert R^{+}(\alpha) \cap \{ (f-c_{R})_+ > \lambda \} \rvert 
\leq A e^{-B(\lambda/\norm{f})^q } \lvert R^{+}(\alpha) \rvert 
\]
and
\[
\lvert R^{-}(\alpha) \cap \{(f-c_{R})_- > \lambda \} \rvert 
\leq  A e^{-B(\lambda/\norm{f})^q } \lvert R^{-}(\alpha) \rvert 
\]
for every $\lambda > 0$.
\end{corollary}

\begin{proof}
By Theorem~\ref{reshetnyak}, there exists a constant $C=C(n,p,q,\gamma,\alpha)$ such that
\begin{align*}
\lvert R^{+}(\alpha)\cap \{(f-c_R)_+ > \lambda \} \rvert 
&\leq e^{-B (\lambda/\norm{f})^q } \frac{A}{\norm{f}^q} \int_{R^{+}(\gamma)} (f-c_R)_+^q \\
&\leq e^{-B(\lambda/\norm{f})^q } \frac{1}{c_2} \frac{1-\alpha}{1-\gamma} \lvert R^{+}(\gamma) \rvert \\
& = \frac{1}{c_2} e^{-B(\lambda/\norm{f})^q} \lvert R^{+}(\alpha) \rvert 
\end{align*}
for every $\lambda \geq C \norm{f}$.
On the other hand, if $0<\lambda < C \norm{f}$, then
\[
\lvert R^{+}(\alpha) \cap \{ (f-c_R)_+ > \lambda \} \rvert 
\leq e^1 e^{-1} \lvert R^{+}(\alpha) \rvert \leq e^1 e^{- \frac{1}{C^q}(\lambda/\norm{f})^q } \lvert R^{+}(\alpha) \rvert .
\]
This proves the first inequality in the theorem. The second inequality follows similarly.
\end{proof}

Another consequence of Theorem~\ref{reshetnyak} is a weak reverse H\"older inequality for parabolic $\BMO$.
In particular, this implies that a function in parabolic $\BMO$ is locally integrable to any positive power.

\begin{corollary}
\label{reverseHolder}
Let $R \subset \mathbb{R}^{n+1}$ be a parabolic rectangle, $0 \leq \gamma < 1$, $\gamma < \alpha < 1$, $0 < q \leq 1$ and $q \leq r < \infty$.
Assume that $f \in\PBMO_{\gamma,q}^+(R)$ and let   $\norm{f}=\norm{f}_{\PBMO_{\gamma,q}^{+}(R)}$. 
Then there exist constants $c_{R}\in\mathbb R$ and $c=c(n,p,q,r,\gamma,\alpha)$ such that
\[
\int_{R^{+}(\alpha)} (f-c_{R})_+^r \leq c \norm{f}^{r-q} \int_{R^{+}(\gamma)} (f-c_{R})_+^q
\]
and
\[
\int_{R^{-}(\alpha)} (f-c_{R})_-^r \leq c \norm{f}^{r-q} \int_{R^{-}(\gamma)} (f-c_{R})_-^q .
\]
\end{corollary}

\begin{proof}
Let
\[
E(\lambda) = R^{+}(\alpha) \cap \{ (f-c_R)_+ > \lambda \} .
\]
By using Cavalieri's principle, we get
\begin{align*}
\int_{R^+(\alpha)} (f-c_R)_+^r  &= r \int_0^\infty \lambda^{r-1} \lvert E(\lambda) \rvert \dla \\
&= r \int_0^{C \norm{f} } \lambda^{r-1} \lvert E(\lambda) \rvert \dla + r \int_{C \norm{f} }^\infty \lambda^{r-1} \lvert E(\lambda) \rvert \dla,
\end{align*}
where $C=C(n,p,q,\gamma,\alpha)$ is the constant in Theorem~\ref{reshetnyak}.
We estimate the obtained integrals separately.
For $0 \leq \lambda \leq C \norm{f} $, we have $\lambda^{r-1} \leq (C \norm{f})^{r-q} \lambda^{q-1}$, and thus
\begin{align*}
r \int_0^{C \norm{f} } \lambda^{r-1} \lvert E(\lambda) \rvert \dla &\leq r (C \norm{f})^{r-q} \int_0^\infty \lambda^{q-1} \lvert E(\lambda) \rvert \dla \\
&= \frac{r}{q} C^{r-q} \norm{f}^{r-q} \int_{R^+(\alpha)} (f-c_R)_+^q .
\end{align*}
For the second integral, we apply Theorem~\ref{reshetnyak} to obtain
\begin{align*}
r \int_{C \norm{f} }^\infty \lambda^{r-1} \lvert E(\lambda) \rvert \dla
 &\leq \frac{Ar}{\norm{f}^q} \int_{R^{+}_0(\gamma)} (f-c_R)_+^q 
 \int_0^\infty \lambda^{r-1} e^{-B (\lambda/\norm{f})^q } \dla \\
&= \frac{Ar}{\norm{f}^q} \int_{R^{+}_0(\gamma)} (f-c_R)_+^q \lp \frac{\norm{f}}{B^\frac{1}{q}} \rp^{r-1} \frac{\norm{f}}{B^\frac{1}{q} q} \int_0^\infty s^{\frac{r}{q}-1} e^{-s} \ds \\
&= \frac{Ar}{B^\frac{r}{q} q} \Gamma\bigl(\tfrac{r}{q}\bigr) \norm{f}^{r-q} \int_{R^{+}_0(\gamma)} (f-c_R)_+^q ,
\end{align*}
where we applied a change of variables $s = B\lambda^q/\norm{f}^q $.
This implies
\[
\int_{R^+(\alpha)} (f-c_R)_+^r \leq c \norm{f}^{r-q} \int_{R^+(\gamma)} (f-c_R)_+^q ,
\quad\text{with}\quad
c = \frac{r}{q} \Bigl( C^{r-q} + \frac{A}{B^\frac{r}{q}} \Gamma\bigl(\tfrac{r}{q}\bigr)\Bigr) .
\]
The second inequality of the theorem follows similarly.
\end{proof}

\section{Chaining arguments and the time lag}

Applying Corollary~\ref{local_pJN} with chaining arguments, we obtain a parabolic John--Nirenberg inequality, which allows us to change the time lag.

\begin{theorem}
\label{global_pJN}
Let $R \subset \mathbb{R}^{n+1}$ be a parabolic rectangle, $0 < \gamma <1$, $-1 < \rho \leq \gamma$, $-\rho < \sigma \leq \gamma$ and $0<q \leq 1$.
Assume that $f \in\PBMO_{\gamma,q}^+(R)$ and let   $\norm{f}=\norm{f}_{\PBMO_{\gamma,q}^{+}(R)}$. 
Then there exist constants $c \in \mathbb{R}$, $A=A(n,p,q,\gamma,\rho,\sigma)$ and $B=B(n,p,q,\gamma,\rho,\sigma)$
such that
\[
\lvert R^{+}(\rho) \cap \{ (f-c)_+ > \lambda \} \rvert \leq A e^{-B (\lambda/\norm{f})^q } \lvert R^{+}(\rho) \rvert
\]
and
\[
\lvert R^{-}(\sigma) \cap \{ (f-c)_- > \lambda \} \rvert \leq  A e^{-B (\lambda/\norm{f})^q } \lvert R^{-}(\sigma) \rvert 
\]
for every $\lambda > 0$.  
\end{theorem}

\begin{proof}
Let $R_0 = R$.
Without loss of generality, we may assume that the center of $R_0$ is the origin.
Since $f \in\PBMO^+_{\gamma,q}(R_0)$, Corollary~\ref{local_pJN} holds for any parabolic subrectangle of $R_0$ and for any $\gamma < \alpha <1$.
Let  $m$ be the smallest integer with
\[
m \geq \log_2\lp \frac{1+\alpha}{1-\alpha} \rp + \frac{1}{p-1} \lp 2 + \log_2 \frac{1+\alpha}{\rho+\sigma} \rp + 2 .
\]
Then there exists $0 \leq \varepsilon < 1$ such that
\[ 
m = \log_2\lp \frac{1+\alpha}{1-\alpha} \rp +  \frac{1}{p-1} \lp 2 + \log_2 \frac{1+\alpha}{\rho+\sigma} \rp + 2 + \varepsilon .
\]
We partition $R^+_0(\rho) = Q(0, L) \times (\rho L^p, L^p)$ by dividing each of its spatial edges into $2^m$ equally long intervals and the time interval into 
$\lceil (1-\rho)2^{mp}/(1-\alpha)\rceil$ equally long intervals.
Denote the obtained rectangles by $U^+_{i,j}$ with $i \in \{1,\dots,2^{mn}\}$ and  $j \in \{1,\dots,\lceil (1-\rho)2^{mp}/(1-\alpha)\rceil\}$.
The spatial side length of $U^+_{i,j}$ is $l = l_x(U^+_{i,j}) =L/2^m$
and the time length is
\[
l_t(U^+_{i,j}) = \frac{(1-\rho)L^p}{\lceil (1-\rho)2^{mp}/(1-\alpha)\rceil} .
\]
For every $U^+_{i,j}$, there exists a unique rectangle $R_{i,j}$ that has the same top as $U^+_{i,j}$.
Our aim is to construct a chain from each $U^+_{i,j}$ to a central rectangle which is of the same form as $R_{i,j}$ and is contained in $R_0$. 
This central rectangle will be specified later.
First, we construct a chain with respect to the spatial variable.
Fix $U^+_{i,j}$.
Let $P_0 = R_{i,j}$ and
\[
P_0^+ = R^+_{i,j}(\alpha) = Q_i \times (t_j - (1-\alpha)l^p, t_j).
\]
We construct a chain of cubes from $Q_i$ to the central cube $Q(0,l)$.
Let $Q'_0 = Q_i = Q(x_i, l)$ and set
\[
Q'_k = Q'_{k-1} - \frac{x_i}{\abs{x_i}} \frac{\theta l}{2}, 
\quad k \in \{0,\dots,N_i\} ,
\]
where $1 \leq \theta \leq \sqrt{n}$ depends on the angle between $x_i$ and the spatial axes and is chosen such that the center of $Q_k$ is on the boundary of $Q_{k-1}$ (see Fig.~\ref{fig:spat_decomp}).
We have
\[
\frac{1}{2^n} \leq \frac{\abs{ Q_k \cap Q_{k-1} }}{\abs{Q_k}} \leq \frac{1}{2}, 
\quad k \in \{0,\dots,N_i\} ,
\]
and $\abs{x_i} = \frac{\theta}{2} (L - bl)$,
where $b \in \{1, \dots, 2^m\}$ depends on the distance of $Q_i$ to the center of $Q_0 = Q(0,L)$.
The number of cubes in the spatial chain $\{Q'_k\}_{k=0}^{N_i}$ is
\[
N_i + 1 = \frac{\abs{x_i}}{\frac{\theta}{2} l} + 1 = \frac{L}{l} - b + 1 .
\]

\begin{figure}[t!]
\centering
\includegraphics[trim=0cm 1.5cm 0cm 0cm, width=1\textwidth]{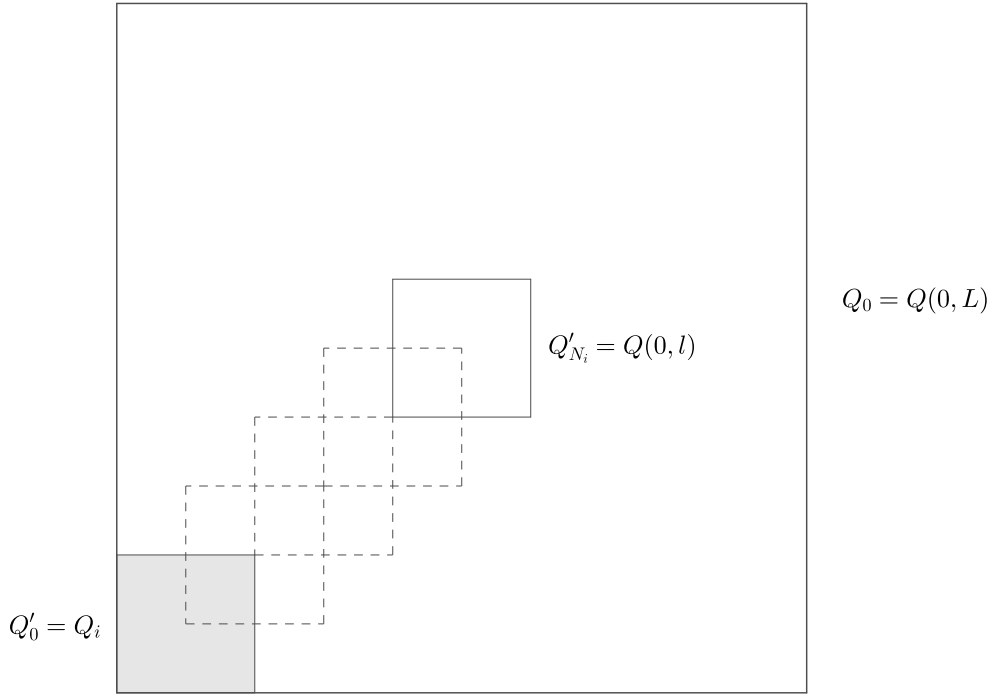}
\caption{Sketch of a spatial chain of cubes $Q'_k$. }
\label{fig:spat_decomp}
\end{figure}

Next, we also take the time variable into consideration in the construction of the chain.
Let
\[
P^+_k = Q'_k \times ( t_j - (1-\alpha)l^p - k(1+\alpha)l^p, t_j - k(1+\alpha)l^p )
\]
and $P^-_k = P^+_k - (0, (1+\alpha)l^p)$,
for $k \in \{ 0, \dots, N_i \}$,
be the upper and the lower parts of a parabolic rectangle respectively.
These will form a chain of parabolic rectangles from $U^+_{i,j}$ to the eventual central rectangle.
Observe that every rectangle $P_{N_i}$ coincides spatially for all pairs $(i,j)$.
Consider $j=1$ and such $i$ that the boundary of $Q_i$ intersects the boundary of $Q_0$.
For such a cube $Q_i$, we have $b=1$, and thus $N = N_i = \frac{L}{l} - 1$.
In the time variable, we travel from $t_1$ the distance
\[
(N+1)(1+\alpha)l^p + (1-\alpha)l^p = (1+\alpha) L l^{p-1} + (1-\alpha) l^p .
\]
We show that the lower part of the final rectangle $P^-_N$ is contained in $R_0$.
To this end, we subtract the time length of $U^+_{i,1}$ from the distance above and observe that it is less than half of the time length of $R_0 \setminus(R_0^+(\rho) \cup R_0^-(\sigma))$.
This follows from the computation
\begin{align*}
&(1+\alpha) L l^{p-1} + (1-\alpha) l^p - \frac{(1-\rho)L^p}{\lceil(1-\rho)2^{mp}/(1-\alpha)\rceil}\\
&\qquad= \lp \frac{1+\alpha}{2^{m(p-1)}} + \frac{1-\alpha}{2^{mp}} - \frac{1-\rho}{\lceil (1-\rho)2^{mp}/(1-\alpha) \rceil} \rp L^p \\
&\qquad\leq \lp \frac{1+\alpha}{2^{m(p-1)}} + \frac{1-\alpha}{2^{mp}} - \frac{1-\rho}{2\frac{(1-\rho)2^{mp}}{1-\alpha}} \rp L^p  \\
&\qquad= \lp \frac{1+\alpha}{2^{m(p-1)}} + \frac{1-\alpha}{2^{mp+1}} \rp L^p 
\leq 2 \frac{1+\alpha}{2^{m(p-1)}} L^p \leq \frac{\rho+\sigma}{2} L^p ,
\end{align*}
since 
\[
m \geq \frac{1}{p-1} \lp 2 + \log_2 \frac{1+\alpha}{\rho+\sigma} \rp.
\]
This implies that $P^-_N \subset R_0^+(\rho - (\rho+\sigma)/2)$.
Denote this rectangle $P_N$ by $\mathfrak{R} = \mathfrak{R}_\rho$.
This is the central rectangle where all chains will eventually end.

Let $j=1$ and assume that $i$ is arbitrary. We extend the chain $\{P_k\}_{k=0}^{N_i}$ by $N - N_i$ rectangles into the negative time direction such that the final rectangle coincides with the central rectangle $\mathfrak{R}$ (see Fig.~\ref{fig:chain_continued}). 
More precisely, we consider $Q'_{k+1} = Q'_{N_i}$, 
\[
P^+_{k+1} = P^+_{k} - (0, (1+\alpha)l^p) 
\quad\text{and}\quad
P^-_{k+1} = P^+_{k+1} - (0, (1+\alpha)l^p)
\]
for $k \in \{N_i, \dots, N-1\}$. 
For every $j \in \{2,\dots,\lceil (1-\rho)2^{mp}/(1-\alpha)\rceil \}$, we consider a similar extension of the chain.
The final rectangles of the chains coincide for fixed $j$ and for every $i$.
Moreover, every chain is of the same length $N+1$, and it holds that
\[
\frac{1}{2^n} \leq \eta = \frac{\lvert P_k^+ \cap P_{k-1}^- \rvert}{\lvert P_k^+ \rvert} \leq 1 .
\]

\begin{figure*}[t!]
    \centering
    \begin{subfigure}[t]{0.49\textwidth}
        \centering
        \includegraphics[trim=3cm -0.5cm 1.5cm -1cm, width=1\textwidth]{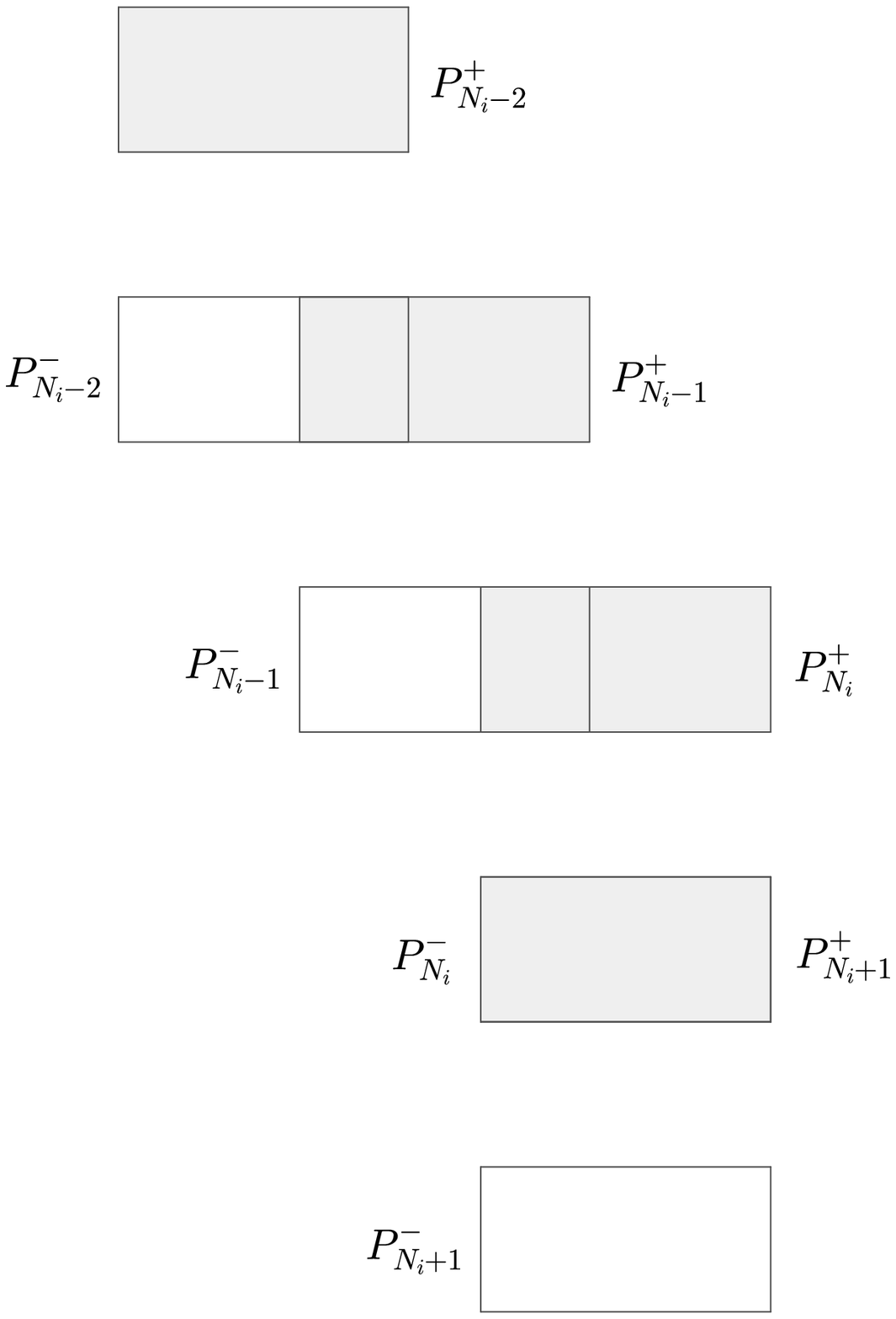}
        \caption{Continued chain (after $N_i$th step).}
        \label{fig:chain_continued}
    \end{subfigure}%
    ~
    \begin{subfigure}[t]{0.49\textwidth}
        \centering
        \includegraphics[trim=1.5cm -1cm 2cm 0cm, width=1\textwidth]{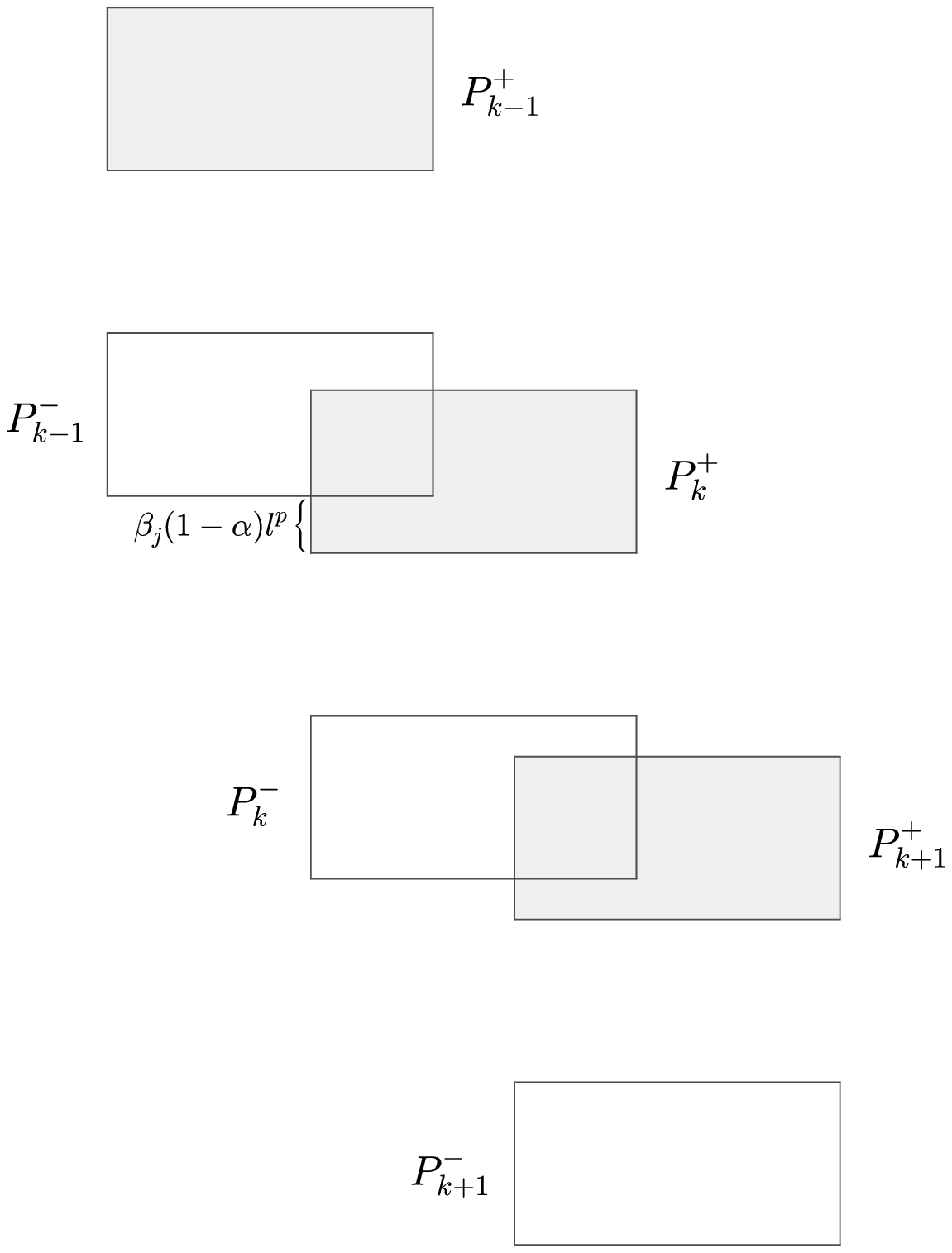}
        \caption{Modified chain.}
        \label{fig:chain_modified}
    \end{subfigure}
    \caption{Schematic pictures of constructed chains.}
\end{figure*}

Then we consider an index $j \in \{2,\dots,\lceil(1-\rho)2^{mp}/(1-\alpha)\rceil \}$ related to the time variable.
The time distance between the current ends of the chains for pairs $(i,j)$ and $(i,1)$ is
\[
(j-1) \frac{(1-\rho) L^p}{\lceil (1-\rho)2^{mp}/(1-\alpha)\rceil} .
\]
Our objective is to have the final rectangle of the continued chain for $(i,j)$ to coincide with the end of the chain for $(i,1)$, that is, with the central rectangle $\mathfrak{R}$.
To achieve this, we modify $2^{m-1}$ intersections of $P^+_k$ and $P^-_{k+1}$ by shifting $P_k$ and also add a chain of $M_j$ rectangles traveling to the negative time direction into the chain $\{P_k\}_{k=0}^{N}$.
We shift every $P_k, k \in \{ 1, \dots, 2^{m-1} \}$, by a $\beta_j$-portion of their temporal length more than the previous rectangle was shifted, that is, we move each $P_k$ into the negative time direction a distance of $k \beta_j (1-\alpha) l^p$ (see Fig.~\ref{fig:chain_modified}).
The values of $M_j \in \mathbb{N}$ and $0\leq \beta_j <1$ will be chosen later.
In other words, modify the definitions of $P^+_k$ for $ k \in \{ 1, \dots, 2^{m-1} \}$ by
\[
P^+_k = Q'_k \times ( t_j - (1-\alpha)l^p - k( 1+\alpha + \beta_j (1-\alpha)) l^p, t_j - k( 1+\alpha + \beta_j (1-\alpha)) l^p) ,
\]
and then add $M_j$ rectangles defined by
\[
P^+_{k+1} = P^+_{k} - (0, (1+\alpha)l^p) 
\quad\text{and}\quad
P^-_{k+1} = P^+_{k+1} - (0, (1+\alpha)l^p) 
\]
for $k \in \{N,\dots, M_j-1\}$.
Note that there exists $1 \leq \tau < 2$ such that 
\[
\tau \frac{(1-\rho)2^{mp}}{1-\alpha} 
=\left\lceil\frac{(1-\rho)2^{mp}}{(1-\alpha)}\right\rceil.
\]
We would like to find such $0\leq \beta_j <1$ and $M_j \in \mathbb{N}$ that 
\[
(j-1) \frac{(1-\rho) L^p}{\lceil (1-\rho)2^{mp}/(1-\alpha)\rceil} - M_j \frac{(1+\alpha) L^p}{2^{mp}} 
= 2^{m-1} \beta_j \frac{(1-\alpha) L^p}{2^{mp}} ,
\]
which is equivalent with
\[
(j-1)\tau^{-1} (1-\alpha) - M_j (1+\alpha) = 2^{m-1} \beta_j (1-\alpha) .
\]
With this choice all final rectangles coincide.
Choose $M_j \in \mathbb{N}$ such that
\[
M_j (1+\alpha) \leq (j-1) \tau^{-1} (1-\alpha) < (M_j + 1)(1+\alpha) ,
\]
that is,
\[
0 \leq \xi = (j-1) \tau^{-1} (1-\alpha) - M_j (1+\alpha) < 1+\alpha 
\]
and
\[
\frac{(j-1)(1-\alpha)}{2 (1+\alpha)} -1 \leq \frac{(j-1)(1-\alpha)}{\tau (1+\alpha)} -1 < M_j \leq \frac{(j-1)(1-\alpha)}{\tau (1+\alpha)} \leq \frac{(j-1)(1-\alpha)}{1+\alpha} .
\]
By choosing $0\leq \beta_j <1$ such that
\[
\xi = 2^{m-1} \beta_j (1-\alpha) = 2^{\frac{2}{p-1} + 1 + \varepsilon} \lp \frac{1+\alpha}{\rho+\sigma} \rp^\frac{1}{p-1} \beta_j (1+\alpha) ,
\]
we have
\[
\beta_j = 2^{-\frac{2}{p-1} - 1 - \varepsilon} \lp \frac{\rho+\sigma}{1+\alpha} \rp^\frac{1}{p-1} \frac{\xi}{1+\alpha} .
\]
Observe that $0\leq \beta_j \leq \frac{1}{2}$ for every $j$.
For measures of the intersections of the modified rectangles, it holds that
\[
\frac{1}{2^{n+1}} \leq \frac{\lvert P_k^+ \cap P_{k-1}^- \rvert}{\lvert P_k^+ \rvert} = \eta (1-\beta_j) \leq 1 
\]
for $ k \in \{ 1, \dots, 2^{m-1} \}$, and thus
\[
\frac{1}{2^{n+1}} \leq \tilde{\eta}_j = \frac{\lvert P_k^+ \cap P_{k-1}^- \rvert}{\lvert P_k^+ \rvert} \leq 1 
\]
for every $k \in \{1,\dots, M_j\}$.
Fix $U^+_{i,j}$. 
We conclude that
\begin{align*}
(c_{R_{i,j}} - c_{\mathfrak{R}} )^q_+ &= (c_{P_0} - c_{P_{N+M_j}} )^q_+  \leq \sum_{k=1}^{N+M_j} (c_{P_{k-1}} - c_{P_{k}} )^q_+ \\
&= \sum_{k=1}^{N+M_j} \dashint_{P_{k-1}^- \cap P_k^+} (c_{P_{k-1}} - c_{P_{k}} )^q_+ \\
&\leq \sum_{k=1}^{N+M_j} \lp \dashint_{P_{k-1}^- \cap P_k^+} (c_{P_{k-1}} - f )^q_+ + \dashint_{P_{k-1}^- \cap P_k^+} (f- c_{P_{k}} )^q_+  \rp \\
&\leq \sum_{k=1}^{N+M_j} \frac{1}{\tilde{\eta}_j} \lp \dashint_{P_{k-1}^-} (f - c_{P_{k-1}})^q_- + \dashint_{P_k^+} (f- c_{P_{k}} )^q_+  \rp \\
&\leq 2^{n+1}  \sum_{k=0}^{N+M_j} \lp \dashint_{P_{k}^-} (f - c_{P_{k}} )^r_- + \dashint_{P_k^+} (f- c_{P_{k}} )^r_+ \rp \\
&\leq 2^{n+1} (N+1+M_j) \norm{f}^q ,
\end{align*}
where
\begin{align*}
N+1+M_j &= \frac{L}{l}+M_j \leq 2^m + (j-1) \frac{1-\alpha}{1+\alpha} \\
&\leq 2^m + \frac{(1-\rho)2^{mp}}{1-\alpha} \frac{1-\alpha}{1+\alpha} 
\leq 2^m + 2^{mp+1} \leq 2^{mp+2} \\
&\leq 2^{\frac{2p}{p-1} + 3p +2} \lp \frac{1+\alpha}{\rho+\sigma} \rp^\frac{p}{p-1} \lp \frac{1+\alpha}{1-\alpha} \rp^p
\end{align*}
for every $j$.
Hence, we obtain
\[
(c_{R_{i,j}} - c_{\mathfrak{R}} )_+ \leq C \norm{f} 
\quad\text{with}\quad
C^q = 2^{\frac{2p}{p-1} + 3p + n +3} \lp \frac{1+\alpha}{\rho+\sigma} \rp^\frac{p}{p-1} \lp \frac{1+\alpha}{1-\alpha} \rp^p.
\]

We observe that
\begin{align*}
&\lvert R^{+}_0(\rho) \cap \{ (f-c_{\mathfrak{R}})_+ > \lambda \} \rvert 
\leq \sum_{i,j} \lvert R^+_{i,j}(\alpha) \cap \{ (f-c_{\mathfrak{R}})_+ > \lambda \} \rvert \\
&\qquad\leq \sum_{i,j} \lvert R^+_{i,j}(\alpha) \cap \{ (f-c_{R_{i,j}})_+ > \tfrac{\lambda}{2} \} \rvert 
 + \sum_{i,j} \lvert R^+_{i,j}(\alpha) \cap \{ (c_{R_{i,j}}-c_{\mathfrak{R}})_+ > \tfrac{\lambda}{2} \} \rvert .
\end{align*}
The first sum of the right-hand side can be estimated by Corollary~\ref{local_pJN} as follows
\begin{align*}
\sum_{i,j} \lvert R^+_{i,j}(\alpha) \cap \{ (f-c_{R_{i,j}})_+ > \tfrac{\lambda}{2} \} \rvert 
&\leq \sum_{i,j} A e^{-B(\lambda/(2\,\norm{f}))^q } \lvert R^+_{i,j}(\alpha) \rvert \\
&\leq A e^{-2^{-q}B(\lambda/\norm{f})^q } \sum_{i,j} \tau \lvert U^+_{i,j} \rvert \\
&\leq 2 A e^{-2^{-q}B(\lambda/\norm{f})^q } \lvert R^+_0(\rho) \rvert .
\end{align*}
To estimate the second sum above, assume that $\lambda \geq 2 C \norm{f}$.
This implies that
\[
\lvert R^+_{i,j}(\alpha) \cap \{ (c_{R_{i,j}}-c_{\mathfrak{R}})_+ > \tfrac{\lambda}{2} \} \rvert 
\leq \lvert R^+_{i,j}(\alpha) \cap \{ C \norm{f} > \tfrac{\lambda}{2} \} \rvert = 0
\]
for every $i,j$.
Thus
\[
\lvert R^{+}_0(\rho) \cap \{ (f-c_{\mathfrak{R}})_+ > \lambda \} \rvert 
\leq 2 A e^{-2^{-q}B (\lambda/\norm{f})^q } \lvert R^+_0(\rho) \rvert 
\]
for every $\lambda \geq 2 C \norm{f}$.
For $0<\lambda < 2 C \norm{f}$, we have
\[
\lvert R^{+}_0(\rho) \cap \{ (f-c_{\mathfrak{R}})_+ > \lambda \} \rvert 
\leq e^1 e^{-1} \lvert R^+_0(\rho) \rvert \leq e^1 e^{-(2C)^{-q}(\lambda/\norm{f})^q } \lvert R^+_0(\rho) \rvert.
\]

We can apply a similar chaining argument in the reverse time direction for $R_0^-(\sigma)$ with the exception that we also extend (and modify if needed) every chain such that the central rectangle $\mathfrak{R}_\sigma$ coincides with $\mathfrak{R}=\mathfrak{R}_\rho$.
A rough upper bound for the number of rectangles needed for the additional extension is given by
\[
\left\lceil \frac{(\rho+\sigma)L^p}{(1+\alpha)l^p} \right\rceil = \left\lceil \frac{\rho+\sigma}{1+\alpha} 2^{mp} \right\rceil \leq 2^{mp+1} .
\]
Thus, the constant $C$ above is two times larger in this case. 
This proves the second inequality of the theorem.
\end{proof}

The next corollary of Theorem~\ref{global_pJN} tells that the spaces $\PBMO_{\gamma,\delta,q}^{+}$ coincide for every $-1<\gamma<1$, $-\gamma < \delta <1$ and $0<q<\infty$.

\begin{corollary}
\label{PBMOequivalent}
Let $\Omega_T \subset \mathbb{R}^{n+1}$ be a space-time cylinder, $0<\gamma<1$, $-1 < \rho \leq \gamma$, $-\rho < \sigma \leq \gamma$ and $0<q \leq r<\infty$.
Then there exist constants $c_1=c_1(n,p,q,r,\gamma,\rho,\sigma)$ and $c_2=c_2(n,p,q,r,\gamma,\rho,\sigma)$ such that 
\[
c_1 \norm{f}_{\PBMO_{\gamma,q}^{+}(\Omega_T)} \leq \norm{f}_{\PBMO_{\rho,\sigma,r}^{+}(\Omega_T)} \leq c_2 \norm{f}_{\PBMO_{\gamma,q}^{+}(\Omega_T)} .
\]
\end{corollary}

\begin{proof}
Let $R$ be a parabolic subrectangle of $\Omega_T$.
By H\"older's inequality, we have
\begin{align*}
&\lp \dashint_{R^+(\gamma)} (f-c_R)_+^q + \dashint_{R^-(\gamma)} (f-c_R)_-^q  \rp^\frac{1}{q} \\
&\qquad\leq \max\{1, 2^{\frac{1}{q}-1} \} \Biggl( \biggl( \dashint_{R^+(\gamma)} (f-c_R)_+^q  \biggr)^\frac{1}{q} + \biggl( \dashint_{R^-(\gamma)} (f-c_R)_-^q \biggr)^\frac{1}{q} \Biggr) \\
&\qquad\leq \max\{1, 2^{\frac{1}{q}-1} \} \Biggl( \biggl( \dashint_{R^+(\gamma)} (f-c_R)_+^r \biggr)^\frac{1}{r} + \biggl( \dashint_{R^-(\gamma)} (f-c_R)_-^r \biggr)^\frac{1}{r} \Biggr) \\
&\qquad\leq c_0 \lp \dashint_{R^+(\gamma)} (f-c_R)_+^r + \dashint_{R^-(\gamma)} (f-c_R)_-^r \rp^\frac{1}{r} ,
\end{align*}
where $c_0 = \max\{1, 2^{\frac{1}{q}-1} \} \max\{1, 2^{1 -\frac{1}{r}} \}$.
We observe that
\begin{align*}
&\lp \dashint_{R^+(\gamma)} (f-c_R)_+^r + \dashint_{R^-(\gamma)} (f-c_R)_-^r \rp^\frac{1}{r}\\
& \qquad\leq \lp \frac{1-\rho}{1-\gamma} \dashint_{R^+(\rho)} (f-c_R)_+^r + \frac{1-\sigma}{1-\gamma} \dashint_{R^-(\sigma)} (f-c_R)_-^r \rp^\frac{1}{r} \\
& \qquad\leq \lp \frac{1-\min\{\rho,\sigma\}}{1-\gamma} \rp^\frac{1}{r} \lp \dashint_{R^+(\rho)} (f-c_R)_+^r + \dashint_{R^-(\sigma)} (f-c_R)_-^r \rp^\frac{1}{r} .
\end{align*}
By taking supremum over all parabolic subrectangles $R \subset \Omega_T$, we arrive at
\[
\norm{f}_{\PBMO_{\gamma,q}^{+}(\Omega_T)} \leq c_0 \lp \frac{1-\min\{\rho,\sigma\}}{1-\gamma} \rp^\frac{1}{r} \norm{f}_{\PBMO_{\rho,\sigma,r}^{+}(\Omega_T)} .
\]

To show the second inequality, we make the restriction $0 < q \leq 1$ so that we can apply Theorem~\ref{global_pJN}. 
This is not an issue since after establishing the second inequality for $0 < q \leq 1$ we can use the first inequality to obtain the whole range $0 < q \leq r$.
Cavalieri's principle and Theorem~\ref{global_pJN} imply that
\begin{align*}
\dashint_{R^+(\rho)} (f-c)_+^r  
&= \frac{r}{\lvert R^+(\rho) \rvert} \int_0^\infty \lambda^{r-1} \lvert R^{+}(\rho) \cap \{ (f-c)_+ > \lambda \} \rvert \dla \\
&\leq A r  \int_0^\infty \lambda^{r-1} e^{-B(\lambda/\norm{f}_{\PBMO_{\gamma,q}^{+}(R)})^q } \dla \\
&= A r \lp \frac{\norm{f}_{\PBMO_{\gamma,q}^{+}(R)}}{B^\frac{1}{q}} \rp^{r-1} \frac{\norm{f}_{\PBMO_{\gamma,q}^{+}(R)}}{B^\frac{1}{q} q} \int_0^\infty s^{\frac{r}{q}-1} e^{-s} \ds \\
&= \frac{Ar}{B^\frac{r}{q} q}  \Gamma\bigl(\tfrac{r}{q}\bigr) \norm{f}_{\PBMO_{\gamma,q}^{+}(R)}^r ,
\end{align*}
where we applied a change of variables $s = B\lambda^q\norm{f}_{\PBMO_{\gamma,q}^{+}(R)}^{-q}$.
Similarly, we obtain
\[
\dashint_{R^-(\sigma)} (f-c)_-^r  \leq \frac{Ar}{B^\frac{r}{q} q}  \Gamma\bigl(\tfrac{r}{q}\bigr) \norm{f}_{\PBMO_{\gamma,q}^{+}(R)}^r .
\]
By adding up the two estimates above and taking supremum over all parabolic rectangles $R \subset \Omega_T$, we conclude that
\[
\norm{f}_{\PBMO_{\rho,\sigma,r}^{+}(\Omega_T)} \leq c_2 \norm{f}_{\PBMO_{\gamma,q}^{+}(\Omega_T)}
\quad\text{with}\quad
c_2 = \lp 2 \frac{Ar}{B^\frac{r}{q} q}  \Gamma\bigl(\tfrac{r}{q}\bigr) \rp^\frac{1}{r} .
\]
\end{proof}

\section{A global parabolic John--Nirenberg inequality} 

The results in the previous sections are local in the sense that they give estimates on a parabolic rectangle $R\subset\Omega_T$.
Next we discuss the corresponding estimates on the entire domain under the assumption that the domain $\Omega\subset\mathbb R^n$
satisfies a quasihyperbolic boundary condition.

\begin{definition}
The quasihyperbolic metric in a domain $\Omega \subset \mathbb{R}^n$ is defined by setting, for any $x_1,x_2\in\Omega$,
\[
k(x_1,x_2) = \inf_{\gamma_{x_1x_2}}\int_{\gamma_{x_1x_2}}\frac{1}{d(x,\partial\Omega)} \ds(x) ,
\]
where the infimum is taken over all rectifiable paths $\gamma_{x_1x_2}$ in $\Omega$ connecting $x_1$ to $x_2$.
\end{definition}

\begin{definition}
A domain $\Omega$ is said to satisfy the quasihyperbolic boundary condition
if there exist $x_0 \in \Omega$ and constants $c_1$ and $c_2$ such that
\[
k(x_0,x) \leq c_1 \log\frac{1}{d(x,\partial\Omega)} + c_2.
\]
for every $x \in \Omega$.
\end{definition}

The class of the domains satisfying the quasihyperbolic boundary condition was first introduced in~\cite{gehringmartio}. 
Note that a domain satisfying the quasihyperbolic boundary condition is bounded.
In~\cite{localtoglobal}, a parabolic John--Nirenberg lemma was proven for domains satisfying the quasihyperbolic boundary condition. We state it here in its complete form.

\begin{theorem}
\label{global_pJN_quasi}
Assume that $\Omega \subset \mathbb{R}^n$ satisfies the quasihyperbolic boundary condition.
Let $0 < \gamma <1$, $0<q \leq 1$ and $0 < \tau_1 < \tau_2 <  T$.
Assume that $f\in\PBMO^+_{\gamma,q}(\Omega_T)$ and let  $\norm{f} = \norm{f}_{\PBMO_{\gamma,q}^{+}(\Omega_T)}$.
Then there exist constants $c \in \mathbb{R}$, $A=A(n,p,q,\gamma,\Omega,\tau_1,\tau_2)$ and $B=B(n,p,q,\gamma,\Omega,\tau_1,\tau_2)$ such that
\[
\lvert\Omega \times (\tau_2,T)\cap \{(f-c)_+ > \lambda \} \rvert \leq A e^{-B(\lambda/\norm{f})^q} \lvert \Omega \times (\tau_2,T) \rvert
\]
and
\[
\lvert\Omega \times (0,\tau_1)\cap\{(f-c)_- > \lambda \} \rvert \leq  A e^{-B(\lambda/\norm{f})^q} \lvert \Omega \times (0,\tau_1) \rvert 
\]
for every $\lambda > 0$.
\end{theorem}

As a corollary of Theorem~\ref{global_pJN_quasi}, the positive part of a PBMO$^{+}$ function is exponentially integrable on upper parts of space-time cylinders and the negative part on lower parts.

\begin{corollary}
\label{exp.int_upper}
Assume that $\Omega \subset \mathbb{R}^n$ satisfies the quasihyperbolic boundary condition and let $0 < \tau_1 < \tau_2 <  T$.
Assume that $f\in\PBMO^+(\Omega_T)$ and let $\norm{f} = \norm{f}_{\PBMO^{+}(\Omega_T)}$.
If $0<\delta<B/\norm{f}$, then
\[
\dashint_{\Omega \times (\tau_2,T)} e^{\delta (f-c)_+} < \infty 
\quad\text{and}\quad
\dashint_{\Omega \times (0,\tau_1)} e^{\delta (f-c)_-} < \infty ,
\]
where $B$ and $c$ are the constants from Theorem~\ref{global_pJN_quasi}.
\end{corollary}

\begin{proof}
Let $E = \Omega \times (\tau_2,T)$.
Cavalieri's principle and Theorem~\ref{global_pJN_quasi} with $q=1$ imply
\begin{align*}
\int_{E} e^{\delta (f-c)_+} 
&= \int_0^\infty \lvert E\cap\{ e^{\delta (f-c)_+} > \lambda \} \rvert \dla \\
&= \int_0^1 \lvert E \cap\{ e^{\delta (f-c)_+} > \lambda \} \rvert \dla 
+ \int_1^\infty \lvert E\cap\{ e^{\delta (f-c)_+} > \lambda \} \rvert \dla \\
&\leq \lvert E \rvert + \int_0^\infty e^s \lvert E\cap\{(f-c)_+ > s/\delta \} \rvert \ds \\
&\leq \lvert E \rvert + A \lvert E \rvert \int_0^\infty e^{s(1 - B/(\delta \norm{f})) } \ds \\
&\le \lvert E \rvert + A \lvert E \rvert \frac{\delta \norm{f}}{B - \delta \norm{f}} < \infty .
\end{align*}
Here we applied a change of variables $\lambda = e^s$. 
The second inequality follows in a similar way. 
\end{proof}

The following theorem gives the characterization of the domains on which the parabolic John--Nirenberg lemma holds globally. 

\begin{theorem}
Let $0 < \tau_1 < \tau_2 < T$.
A domain $\Omega$ satisfies the quasihyperbolic boundary condition if and only if there exist $\delta > 0$ and $c\in\mathbb{R}$ such that 
\[
\dashint_{\Omega \times (\tau_2,T)} e^{\delta (f-c)_+/\norm{f}}\leq 2
\quad\text{and}\quad
\dashint_{\Omega \times (0,\tau_1)} e^{\delta (f-c)_-/\norm{f}}\leq 2
\]
for every $f \in \PBMO^+(\Omega_T)$ with $\norm{f} = \norm{f}_{\PBMO^{+}(\Omega_T)}$.
\end{theorem}

\begin{proof}
One direction follows from the proof of Corollary~\ref{exp.int_upper} by choosing $0<\delta \leq \tfrac{B}{1+A} \frac{1}{\norm{f}}$.
For the other direction, we consider $f(x,t) = k(x_0,x)$, where $k(x_0,x)$ denotes the quasihyperbolic metric in $\Omega$.
We note that 
\[
\norm{f}_{\PBMO^{+}(\Omega_T)} \leq \norm{f}_{\BMO(\Omega)} \leq 2 \norm{f}_{\PBMO^{+}(\Omega_T)}.
\]
By~\cite[Theorem A]{smithstegenga}, we conclude that $f \in\BMO(\Omega)$. Thus, we have $f \in \PBMO^+(\Omega_T)$ and the parabolic John--Nirenberg inequality of the claim applies for $f$.
Set $\tilde{\delta} =\delta/4$.
By applying Jensen's inequality twice, Young's inequality and the parabolic John--Nirenberg inequality, we obtain
\begin{align*}
\dashint_{\Omega} e^{\tilde{\delta} \lvert f-f_\Omega \rvert/\norm{f}_{\BMO(\Omega)}}
&\leq \dashint_{\Omega} e^{\tilde{\delta} \lvert f-f_\Omega \rvert/\norm{f}}
\leq e^{\tilde{\delta} \lvert f_\Omega - c \rvert/\norm{f}} 
\dashint_{\Omega} e^{\tilde{\delta} \lvert f-c \rvert/\norm{f}}\\
&\leq e^{\dashint_\Omega\tilde{\delta} \lvert f- c \rvert/\norm{f}} 
\dashint_{\Omega} e^{\tilde{\delta} \lvert f-c \rvert/\norm{f}}
\leq \lp \dashint_{\Omega} e^{\tilde{\delta} \lvert f-c \rvert/\norm{f}}\rp^2 \\
&\leq \dashint_{\Omega} e^{2 \tilde{\delta}\lvert f-c \rvert/\norm{f}}
= \dashint_{\Omega} e^{2\tilde{\delta} (f-c)_+/\norm{f}}e^{2\tilde{\delta} (f-c)_-/\norm{f}}\\
&\leq \frac{1}{2} \dashint_{\Omega} e^{4\tilde{\delta} (f-c)_+/\norm{f}} 
+ \frac{1}{2} \dashint_{\Omega} e^{4\tilde{\delta} (f-c)_-/\norm{f}}\\
&= \frac{1}{2} \dashint_{\Omega \times (\tau_2,T)} e^{\delta (f-c)_+/\norm{f}} 
+\frac{1}{2} \dashint_{\Omega \times (0,\tau_1)} e^{\delta (f-c)_-/\norm{f}}
\leq 2 .
\end{align*}
By~\cite[Theorem A]{smithstegenga}, the domain $\Omega$ satisfies the quasihyperbolic boundary condition.
\end{proof}

\section{Parabolic $\BMO$ with medians} 
\label{section5}

This section discusses John--Nirenberg inequalities for the median-type parabolic $\BMO$.
In many cases, it is preferable to consider medians instead of integral averages.
Let $0<s \leq 1$. Assume that $E \subset \mathbb{R}^{n+1}$ is a measurable set with $0<|E|<\infty$ and that $f:E\rightarrow [-\infty, \infty]$ is a measurable function.
A number $a\in\mathbb R$ is called an $s$-median of $f$ over $E$, if
\[
|\{x \in E: f(x) > a\}| \le s |E|
\quad\text{and}\quad
|\{x \in E: f(x) < a\}| \le (1-s)|E|.
\]
In general, the $s$-median is not unique. To obtain a uniquely defined notion, we consider the maximal $s$-median as in \cite{medians}.

\begin{definition}
\label{maxmedian}
Let $0<s \leq 1$. Assume that $E \subset \mathbb{R}^{n+1}$ is a measurable set with $0<|E|<\infty$ and that $f:E\rightarrow [-\infty, \infty]$ is a measurable function.
The maximal $s$-median of $f$ over $E$ is defined as
\[
m_f^s(E) = \inf \{a \in \mathbb{R} : |\{x \in E: f(x) > a\}| < s |E| \} .
\]
\end{definition}

The maximal $s$-median of a function is an $s$-median~\cite{medians}.
In the next lemma, we list the basic properties of the maximal $s$-median.
We refer to~\cite{myyrylainen,medians} for the proofs of the properties.

\begin{lemma}
\label{medianprops}
Let $0<s \leq 1$. Assume that $E \subset \mathbb{R}^{n+1}$ is a measurable set with $0<|E|<\infty$ and that $f,g:E\rightarrow [-\infty, \infty]$ are measurable functions.
The maximal $s$-median has the following properties.
\begin{enumerate}[(i),topsep=5pt,itemsep=5pt]

\item$m_f^{s'}(E) \leq m_f^s(E)$ for any $0<s \leq s'\le1$.

\item $m_f^s(E) \leq m_g^s(E)$ whenever $f\leq g$ almost everywhere in $E$.

\item If $E \subset E'$ and $|E'| \leq c |E|$ with some $c \geq 1$, then $m_f^s(E) \leq m_f^{s/c}(E')$.

\item $m_{\varphi \circ f}^s(E) = \varphi(m_f^s(E))$ for an increasing continuous function $\varphi: f(E) \to [-\infty, \infty]$.

\item $m_f^s(E) + c = m_{f+c}^s(E)$ for any $c \in \mathbb{R}$.

\item $m_{cf}^s(E) = c \, m_f^s(E)$ for any $c > 0$.

\item $|m_{f}^s(E)| \leq m_{|f|}^{\min\{s,1-s\}}(E)$.

\item $m_{f+g}^s(E) \leq m_f^{t_1}(E) + m_g^{t_2}(E)$ whenever $t_1 + t_2 \leq s$.

\item For any $f \in L^p(E)$ with $p>0$, \[ m_{|f|}^s(E) \leq \lp s^{-1} \dashint_E |f|^p \rp^{\frac{1}{p}}. \]

\item If $E_i$,  $i \in \mathbb{N}$, are pairwise disjoint measurable sets, then
\[
\inf_{i} m_f^s(E_i) \leq m_f^s\Bigl(\bigcup_{i=1}^\infty E_i\Bigr) \leq \sup_{i} m_f^s(E_i) .
\]
\end{enumerate}
\end{lemma}

Lemma~\ref{LDT} can be combined with the proof of the Lebesgue differentiation theorem for medians in~\cite{medians} to obtain the following lemma. 

\begin{lemma}
\label{leb.diff.medians}
Let $f: \mathbb{R}^{n+1} \rightarrow [-\infty, \infty]$ be a measurable function which is finite almost everywhere in $\mathbb R^{n+1}$ and $0<s\leq 1$. Then 
\[
\lim_{i \to \infty}  m_f^s(A_i) = f(x,t)
\]
for almost every $(x,t) \in \mathbb{R}^{n+1}$, whenever $(A_i)_{i\in\mathbb N}$ is a sequence of measurable sets converging regularly to $(x,t)$.
\end{lemma}

\begin{definition}
Let $\Omega \subset \mathbb{R}^n$ be a domain and $T>0$.
Given $0 \leq \gamma < 1$ and $0<s \leq1$, we say that a measurable function $f:\Omega_T\to[-\infty,\infty]$ belongs to the median-type parabolic $\BMO$, denoted by $\PBMO_{\gamma,0,s}^{+}(\Omega_T)$, if
\[
\norm{f}_{\PBMO_{\gamma,0,s}^{+}(\Omega_T)} 
= \sup_{R \subset \Omega_T} \inf_{c \in \mathbb{R}} \lp m_{(f-c)_+}^s(R^+(\gamma)) + m_{(f-c)_-}^s(R^-(\gamma)) \rp < \infty .
\]
If the condition above holds with the time axis reversed, then $f \in\PBMO_{\gamma,0,s}^{-}(\Omega_T)$.
\end{definition}

The next lemma is a counterpart of Lemma~\ref{PBMO_constant}. The proof is similar to that of Lemma~\ref{PBMO_constant} and thus is omitted here.

\begin{lemma}\label{lemma.mpbmonorm}
Let $\Omega_T \subset \mathbb{R}^{n+1}$ be a space-time cylinder, $0 \leq \gamma < 1$ and $0<s \leq1$. Assume that $f:\Omega_T\to[-\infty,\infty]$ is a measurable function. 
Then for every parabolic rectangle $R\subset\Omega_T$, there exists a constant $c_R\in\mathbb R$, that may depend on $R$, such that
\[
m_{(f-c_R)_+}^s(R^+(\gamma)) + m_{(f-c_R)_-}^s(R^-(\gamma)) 
= \inf_{c \in \mathbb{R}} \lp m_{(f-c)_+}^s(R^+(\gamma)) + m_{(f-c)_-}^s(R^-(\gamma)) \rp .
\]
In particular,
\[
\sup_{R \subset \Omega_T} \lp m_{(f-c_R)_+}^s(R^+(\gamma)) + m_{(f-c_R)_-}^s(R^-(\gamma)) \rp = \norm{f}_{\PBMO_{\gamma,0,s}^{+}(\Omega_T)} .
\]
\end{lemma}

The following John--Nirenberg lemma is a counterpart of Corollary~\ref{local_pJN}. We apply the same decomposition argument as in the proof of Theorem~\ref{reshetnyak}.

\begin{theorem}
\label{local_pJN_m}
Let $R \subset \mathbb{R}^{n+1}$ be a parabolic rectangle, $0 \leq \gamma < 1$, $\gamma < \alpha < 1$ and $0 < s \leq s_0$, where $s_0$ is a small positive number.
Assume that $f \in \PBMO_{\gamma,0,s}^+(R)$ and let   $\norm{f} = \norm{f}_{\PBMO_{\gamma,0,s}^{+}(R)}$.
Then there exist constants $c_R\in\mathbb R$, $A=A(n,p,\gamma,\alpha)$ and $B=B(n,p,\gamma,\alpha)$ such that
\[
\lvert R^{+}(\alpha) \cap \{ (f-c_R)_+ > \lambda \} \rvert \leq A e^{-B\lambda/\norm{f}} \lvert R^{+}(\alpha) \rvert 
\]
and
\[
\lvert R^{-}(\alpha) \cap \{ (f-c_R)_- > \lambda \} \rvert \leq  A e^{-B\lambda/\norm{f}} \lvert R^{-}(\alpha) \rvert 
\]
for every $\lambda > 0$.
\end{theorem}

\begin{proof}
We use the same notation as in the proof of Theorem~\ref{reshetnyak} until equation~\eqref{start} with the exception that we assume $\norm{f} = 1$.
We proceed from there.
It holds that
\begin{equation}
\label{start_m}
\lvert S(\lambda) \rvert = \Big\lvert \bigcup_i S^+_i \Big\rvert \leq c_1 \sum_j \lvert \widetilde{S}^-_j \rvert ,
\end{equation}
where $c_1 = 3(7+\alpha)/(1-\alpha)$.
Take $\lambda > \delta > 0$ and form $\{S^+_k\}_k$ for $\delta$.
Each $S^+_i$ is contained in a unique $S^+_k$.
Set $\mathcal{J}_k = \{ j: \widetilde{S}^+_j \subset S^+_k \}$ and 
define 
\[\mathcal{K} = \Bigl\{ k \in \mathbb{N} : \lvert S^+_k \rvert \leq 2 c_1 \sum_{j \in \mathcal{J}_k} \lvert \widetilde{S}^-_j \rvert \Bigr\} .
\]
By using properties (viii) and (iii) of Lemma~\ref{medianprops} together with~\eqref{measure2}, we obtain
\begin{align*}
\lambda &< c_{R_j} = m_{c_{R_j}}^{1}(\widetilde{S}^-_j) \leq
m_{(f-c_{R_j})_-}^{2s(1-\gamma)/(1-\alpha)}(\widetilde{S}^-_j) + m_{f_+}^{r}(\widetilde{S}^-_j) \\
&\leq m_{(f-c_{R_j})_-}^{s}(R^-_j(\gamma)) + m_{f_+}^{r}(\widetilde{S}^-_j) \\
&\leq \norm{f} + m_{f_+}^{r}(\widetilde{S}^-_j) = 1 + m_{f_+}^{r}(\widetilde{S}^-_j)
\end{align*}
for every $\widetilde{S}^-_j$, where $r=1 - 2s(1-\gamma)/(1-\alpha)$.
Since $\widetilde{S}^-_j$ are pairwise disjoint for $j \in \mathcal{J}_k$, property (x) of Lemma~\ref{medianprops} implies that
\[
\lambda - 1 \leq m_{f_+}^{r}\Bigl( \bigcup_{j \in \mathcal{J}_k} \widetilde{S}^-_j \Bigr)
\]
for every $k \in \mathbb{N}$.

Fix $k \in \mathcal{K}$. We have $\widetilde{S}^+_j \subset S^+_k$ for all $j \in \mathcal{J}_k$, where $S^+_k$ was obtained by subdividing a previous $S^+_{k^-}$ for which $a_{R_{k^-}} \leq \delta$. Hence, it holds that $\widetilde{S}^-_j \subset R_k$ for all $j \in \mathcal{J}_k$.
By \eqref{subset}, it follows that $\widetilde{S}^-_j \subset R^+_{k^-}(\gamma)$ for every $j \in \mathcal{J}_k$ and thus also $\bigcup_{j \in \mathcal{J}_k} \widetilde{S}^-_j \subset R^+_{k^-}(\gamma)$.
Using~\eqref{measure1} and~\eqref{measure2}, we get
\begin{equation}
\label{med_measure}
\lvert R^+_{k^-}(\gamma) \rvert 
\leq 2 \frac{1-\gamma}{1-\alpha} \lvert S^+_{k^-} \rvert \leq 2 \frac{1-\gamma}{1-\alpha} 2^{nm} \lceil 2^{pm} \rceil \lvert S^+_k \rvert
\leq c_2 \lvert S^+_k \rvert \leq 2 c_1 c_2 \Bigl\lvert \bigcup_{j \in \mathcal{J}_k} \widetilde{S}^-_j \Bigr\rvert
\end{equation}
for every $k \in \mathcal{K}$,
where
\[
c_2 = \frac{1-\gamma}{1-\alpha} 2^{2+n+p} \lp\frac{3+\alpha}{2(\alpha- \gamma)}\rp^{1 + \frac{n}{p}} .
\]
By applying (iii), (ii) and (v) of  Lemma~\ref{medianprops}, we have
\begin{equation}
\label{med_lambda}
\lambda - 1 
\leq m_{f_+}^{r}\Bigl( \bigcup_{j \in \mathcal{J}_k} \widetilde{S}^-_j \Bigr) \leq m_{f_+}^{r/\tilde{c}}( R^+_{k^-}(\gamma) ) \\
\leq m_{(f - a_{R_{k^-}})_+}^{r/\tilde{c}}( R^+_{k^-}(\gamma) ) + \delta 
\leq 1 + \delta 
\end{equation}
for $k \in \mathcal{K}$, where $\tilde{c} = 2 c_1 c_2$ and whenever $s \leq r / \tilde{c}$.
By combining estimates~\eqref{med_measure} and~\eqref{med_lambda}, we obtain
\[
(\lambda - \delta - 1) \sum_{j \in \mathcal{J}_k} \lvert \widetilde{S}^-_j \rvert \leq \lvert R^+_{k^-}(\gamma) \rvert \leq c_2 \lvert S^+_k \rvert 
\]
for $k \in \mathcal{K}$. Thus, whenever $\lambda > \delta + 1$, we have
\[
\sum_{k \in \mathcal{K}} \sum_{j \in \mathcal{J}_k} \lvert \widetilde{S}^-_j \rvert \leq \frac{c_2}{\lambda - \delta -1} \sum_{k \in \mathcal{K}} \lvert S^+_k \rvert .
\]

On the other hand, if $k \notin \mathcal{K}$, we have
\[
\sum_{j \in \mathcal{J}_k} \lvert \widetilde{S}^-_j \rvert \leq \frac{1}{2c_1} \lvert S^+_k \rvert ,
\]
which implies
\[
\sum_{k \notin \mathcal{K}} \sum_{j \in \mathcal{J}_k} \lvert \widetilde{S}^-_j \rvert \leq \frac{1}{2c_1} \sum_{k \notin \mathcal{K}} \lvert S^+_k \rvert .
\]
By combining the cases $k \in \mathcal{K}$ and $k \notin \mathcal{K}$, we obtain
\[
\sum_{j}  \lvert \widetilde{S}^-_j \vert = \sum_{k} \sum_{j \in \mathcal{J}_k} \lvert \widetilde{S}^-_j \rvert \leq \frac{c_2}{\lambda - \delta -1} \sum_{k \in \mathcal{K}} \lvert S^+_k \rvert + \frac{1}{2c_1} \sum_{k \notin \mathcal{K}} \lvert S^+_k \rvert ,
\]
whenever $\lambda > \delta + 1$.
Applying~\eqref{start_m}, we arrive at
\[
\lvert S(\lambda) \rvert \leq \frac{c_1 c_2}{\lambda - \delta -1} \sum_{k \in \mathcal{K}} \lvert S^+_k \rvert + \frac{1}{2} \sum_{k \notin \mathcal{K}} \lvert S^+_k \rvert ,
\]
whenever $\lambda > \delta + 1$.
Set $a = 2c_1 c_2 + 1$ and replace $\lambda$ and $\delta$ by $\lambda + a$ and $\lambda$, respectively.
We have
\begin{equation}
\label{med_iteration}
\lvert S(\lambda + a) \rvert \leq \frac{1}{2} \sum_{k \in \mathcal{K}} \lvert S^+_k \rvert + \frac{1}{2} \sum_{k \notin \mathcal{K}} \lvert S^+_k \rvert = \frac{1}{2} \sum_k \lvert S^+_k \rvert \leq \frac{1}{2} \lvert S(\lambda) \rvert .
\end{equation}
Assume that $\lambda \geq a$. Then there exists an integer $N \in \mathbb{Z}_+$ such that $Na \leq \lambda < (N+1)a$.
A recursive application of~\eqref{med_iteration} gives
\[
\lvert S(\lambda) \rvert 
\leq \lvert S(Na) \rvert \leq \frac{1}{2^{N-1}} \lvert S(a) \rvert 
\leq 2^{-\frac{\lambda}{a}+2} \lvert S_0^+ \rvert = 4 e^{-\frac{\lambda}{2c_1 c_2 +1} \log 2} \lvert S_0^+ \rvert .
\]
Hence, we have
\[
\lvert S(\lambda) \rvert \leq A e^{-2B\lambda/\norm{f}}  \lvert S_0^+ \rvert 
\]
for $\lambda \geq a$, where $A=4$ and $B = \frac{1}{2} \log 2/(2 c_1 c_2 + 1)$.

If $(x,t) \in S^+_0 \setminus S(\lambda)$, then there exists a sequence $\{S^+_l\}_{l\in\mathbb N}$ of subrectangles  containing $(x,t)$ such that $c_{R_l} \leq \lambda $ and $\lvert S^+_l \rvert \to 0$ as $l \to \infty$.
By (ii) and (v) of Lemma~\ref{medianprops}, we have
\[
m_{f_+}^{2s(1-\gamma)/(1-\alpha)}(S^+_l) \leq  m_{(f-c_{R_l})_+}^{2s(1-\gamma)/(1-\alpha)}(S^+_l) + \lambda \leq 1 + \lambda .
\]
Lemma~\ref{leb.diff.medians} then further implies that
$f(x,t)_+ \leq 1 + \lambda$ for almost every $(x,t) \in S^+_0 \setminus S(\lambda)$.
It follows that
\[
\{ (x,t) \in S^+_0 : f(x,t)_+ > 1 + \lambda \} \subset S(\lambda)
\]
up to a set of measure zero.
For $\lambda \geq 2$, we have $\lambda \geq 1 + \frac{\lambda}{2}$. We conclude that
\[
\lvert S^+_0 \cap \{ f_+ > \lambda \} \rvert 
\leq \lvert S^+_0 \cap \{ f_+ > 1 + \tfrac{\lambda}{2} \} \rvert 
\leq \lvert S(\tfrac{\lambda}2) \rvert 
\leq A e^{-B\lambda/\norm{f}} \lvert S_0^+ \rvert 
\]
for every $\lambda \geq 2a$.
If $0<\lambda < 2a$, the claim follows from the estimate
\[
\lvert S^+_0 \cap \{ f_+ > \lambda \} \rvert \leq e^1 e^{-1} \lvert S^+_0 \rvert \leq e^1 e^{-\frac{1}{2a}\lambda/\norm{f}} \lvert S^+_0 \rvert .
\]

Finally, we discuss the restriction on the median level parameter $s$,
\[
s \leq \frac{r}{\tilde{c}} = \frac{1 - 2\tfrac{1-\gamma}{1-\alpha} s}{\tilde{c}},
\quad\text{that is,}\quad
s \leq \frac{1}{\tilde{c} + 2\tfrac{1-\gamma}{1-\alpha}} = s_0.
\]
The proof is complete. 
\end{proof}

The following John--Nirenberg inequality is an analogy of Theorem~\ref{global_pJN} and its proof uses the same chaining argument.

\begin{theorem}
\label{global_pJN_m}
Let $R \subset \mathbb{R}^{n+1}$ be a parabolic rectangle, $0 < \gamma <1$, $-1 < \rho \leq \gamma$, $-\rho < \sigma \leq \gamma$ and $0<s \leq s_0$.
Assume that $f \in \PBMO_{\gamma,0,s}^+(R)$ and let  $\norm{f} = \norm{f}_{\PBMO_{\gamma,0,s}^{+}(R)}$.
Then there exist constants $c \in \mathbb{R}$, $A=A(n,p,\gamma,\rho,\sigma)$ and $B=B(n,p,\gamma,\rho,\sigma)$ such that 
\[
\lvert R^{+}(\rho) \cap \{ (f-c)_+ > \lambda \} \rvert \leq A e^{-B(\lambda/\norm{f})^q } \lvert R^{+}(\rho) \rvert
\]
and
\[
\lvert R^{-}(\sigma) \cap \{ (f-c)_- > \lambda \} \rvert 
\leq  A e^{-B (\lambda/\norm{f})^q } \lvert R^{-}(\sigma) \rvert 
\]
for every $\lambda > 0$.
\end{theorem}

\begin{proof}
We use the same notation as in the proof of Theorem~\ref{global_pJN} until the point where $(c_{R_{i,j}} - c_{\mathfrak{R}} )^q_+$ is estimated.
By (v), (viii), (iii) and (i)  of  Lemma~\ref{medianprops} in this order, we obtain
\begin{align*}
(c_{R_{i,j}} - c_{\mathfrak{R}} )_+ 
&= (c_{P_0} - c_{P_{N+M_j}} )_+  
\leq \sum_{k=1}^{N+M_j} (c_{P_{k-1}} - c_{P_{k}} )_+\\
&= \sum_{k=1}^{N+M_j} m_{(c_{P_{k-1}} - c_{P_{k}} )_+}^{1}(P_{k-1}^- \cap P_k^+) \\
&\leq \sum_{k=1}^{N+M_j} \lp  m_{(c_{P_{k-1}} - f )_+}^{1/2}(P_{k-1}^- \cap P_k^+) +  m_{(f- c_{P_{k}} )_+}^{1/2}(P_{k-1}^- \cap P_k^+)  \rp \\
&\leq \sum_{k=1}^{N+M_j} \lp  m_{(f - c_{P_{k-1}} )_-}^{\tilde{\eta}_j /2}(P_{k-1}^-) +  m_{(f- c_{P_{k}} )_+}^{\tilde{\eta}_j /2}(P_k^+)  \rp \\
&\leq \sum_{k=0}^{N+M_j} \lp m_{(f - c_{P_{k}} )_-}^{ 1/2^{n+2}}(P_{k}^-) +  m_{(f- c_{P_{k}} )_+}^{ 1/2^{n+2}}(P_k^+) \rp \\
&\leq (N+1+M_j) \norm{f} ,
\end{align*}
whenever $s \leq1/2^{n+2}$. 
This is satisfied by the assumption $s \leq s_0$.
We observe that
\begin{align*}
N+1+M_j &= \frac{L}{l}+M_j \leq 2^m + (j-1) \frac{1-\alpha}{1+\alpha} \\
&\leq 2^m + \frac{(1-\rho)2^{mp}}{1-\alpha} \frac{1-\alpha}{1+\alpha} 
\leq 2^m + 2^{mp+1} \leq 2^{mp+2} \\
&\leq 2^{\frac{2p}{p-1} + 3p +2} \lp \frac{1+\alpha}{\rho+\sigma} \rp^\frac{p}{p-1} \lp \frac{1+\alpha}{1-\alpha} \rp^p = C
\end{align*}
for every $j$.
Hence, it holds that $(c_{R_{i,j}} - c_{\mathfrak{R}} )_+ \leq C \norm{f} $.
The proof now proceeds in the exactly same way as in Theorem~\ref{global_pJN} except applying Theorem~\ref{local_pJN_m} instead of Corollary~\ref{local_pJN}. Thus, we can stop here.
\end{proof}

As a corollary of Theorem~\ref{global_pJN_m}, the median-type parabolic $\BMO$ coincides with the classical integral-type parabolic $\BMO$, compare with Corollary~\ref{PBMOequivalent}.
In particular, it follows that all results for the integral-type parabolic $\BMO$ also hold for the median-type parabolic $\BMO$ and vice versa.

\begin{corollary}
\label{PBMOequivalent_m}
Let $\Omega_T \subset \mathbb{R}^{n+1}$ be a space-time cylinder, $0 < \gamma <1$, $0<q<\infty$, $0 < \rho < 1$ and $0 < s \leq s_0 $. 
Then there exist constants $c_1=c_1(n,p,q,\gamma,\rho,s)$ and $c_2=c_2(n,p,q,\gamma,\rho,s)$ such that 
\[
c_1 \norm{f}_{\PBMO_{\gamma,0,s}^{+}(\Omega_T)} \leq \norm{f}_{\PBMO_{\rho,q}^{+}(\Omega_T)} \leq c_2 \norm{f}_{\PBMO_{\gamma,0,s}^{+}(\Omega_T)} .
\]
\end{corollary}

\begin{proof}
Fix $0 < \gamma < 1$ and assume that $0 < \rho \leq \gamma$.
Let $R$ be a parabolic subrectangle of $\Omega_T$.
By Lemma~\ref{medianprops} (ix), we have
\begin{align*}
&m_{(f-c_{R})_+}^{s}(R^+(\gamma  )) + m_{(f-c_{R})_-}^{s}(R^-(\gamma)) \\
&\qquad\leq \lp s^{-1} \dashint_{R^+(\gamma)} (f-c_R)_+^q  \rp^{\frac{1}{q}} + \lp s^{-1} \dashint_{R^-(\gamma)} (f-c_R)_-^q  \rp^{\frac{1}{q}} \\
&\qquad\leq c_0 \lp \dashint_{R^+(\gamma)} (f-c_R)_+^q  + \dashint_{R^-(\gamma)} (f-c_R)_-^q \rp^\frac{1}{q} ,
\end{align*}
where $c_0 = s^{-\frac{1}{q}} \max\{1, 2^{1 -\frac{1}{q}} \}$.
We observe that
\begin{align*}
&\lp \dashint_{R^+(\gamma)} (f-c_R)_+^q + \dashint_{R^-(\gamma)} (f-c_R)_-^q \rp^\frac{1}{q} \\
&\qquad\leq \lp \frac{1-\rho}{1-\gamma} \rp^\frac{1}{q} \lp \dashint_{R^+(\rho)} (f-c_R)_+^q + \dashint_{R^-(\rho)} (f-c_R)_-^q \rp^\frac{1}{q} .
\end{align*}
Taking supremum over all parabolic subrectangles $R \subset \Omega_T$, we obtain
\[
\norm{f}_{\PBMO_{\gamma,0,s}^{+}(\Omega_T)} \leq c_0 \lp \frac{1-\rho}{1-\gamma} \rp^\frac{1}{q} \norm{f}_{\PBMO_{\rho,q}^{+}(\Omega_T)}
\]
for $0 < \rho \leq \gamma$.

For the second inequality, Cavalieri's principle and Theorem~\ref{global_pJN_m} imply that
\begin{align*}
\dashint_{R^+(\rho)} (f-c)_+^q 
&= \frac{q}{\lvert R^+(\rho) \rvert} \int_0^\infty \lambda^{q-1} \lvert R^{+}(\rho) \cap \{ (f-c)_+ > \lambda \} \rvert \dla \\
&\leq A q  \int_0^\infty \lambda^{q-1} e^{-B \lambda/\norm{f}_{\PBMO_{\gamma,0,s}^{+}(R)} } \dla \\
&= A q \lp \frac{\norm{f}_{\PBMO_{\gamma,0,s}^{+}(R)}}{B} \rp^{q-1} \frac{\norm{f}_{\PBMO_{\gamma,0,s}^{+}(R)}}{B} \int_0^\infty s^{q-1} e^{-s} \ds \\
&= \frac{Aq}{B^q} \Gamma(q) \norm{f}_{\PBMO_{\gamma,0,s}^{+}(R)}^q , 
\end{align*}
where we made a change of variables $s = B\lambda/\norm{f}_{\PBMO_{\gamma,0,s}^{+}(R)}$.
Similarly, we obtain
\[
\dashint_{R^-(\rho)} (f-c)_-^q \leq \frac{Aq}{B^q}  \Gamma(q) \norm{f}_{\PBMO_{\gamma,0,s}^{+}(R)}^q .
\]
By summing the two estimates above and taking supremum over all parabolic rectangles $R \subset \Omega_T$, we get
\[
\norm{f}_{\PBMO_{\rho,q}^{+}(\Omega_T)} \leq c_2 \norm{f}_{\PBMO_{\gamma,0,s}^{+}(\Omega_T)}
\quad\text{with}\quad
c_2 = \lp 2 \frac{Aq}{B^q}  \Gamma(q) \rp^\frac{1}{q}
\]
for $0 < \rho \leq \gamma$.
Applying Corollary~\ref{PBMOequivalent}, we obtain the claim for the whole range $0 < \rho  < 1$.
\end{proof}

\end{document}